\documentclass[10pt, oneside]{amsart}   	
\usepackage[margin=1in]{geometry}                		
\usepackage{cite}
\usepackage{graphicx}				
\usepackage{amssymb}

\usepackage{hyperref}

\usepackage{amssymb,bbm}
\usepackage{amsthm}
\usepackage{graphicx}
\usepackage{mathtools}
\usepackage{amsmath,amsfonts}
\usepackage{epsfig}
\usepackage{comment}
\usepackage{color}
\usepackage{amsmath}

\usepackage{todonotes}

\usepackage{float}
\usepackage{tikz}
\usetikzlibrary{shapes.geometric}

\newcommand{\stwedge}{\overline{\wedge}}

\newcommand{\Mod}{{\rm Mod}}

\newcommand{\image}{{\rm Image}}

\newcommand{\Z}{{\mathbb Z}}
\newcommand{\R}{{\mathbb R}}

\newcommand{\N}{{\mathbb N}}

\newcommand{\cB}{{\mathcal B}}
\newcommand{\cV}{{\mathcal V}}
\newcommand{\cC}{{\mathcal C}}

\newcommand{\cH}{\mathcal{H}}
\newcommand{\cP}{\mathcal{P}}
\newcommand{\cU}{\mathcal{U}}

\newcommand{\bfB}{\textbf{B}}

\newcommand{\dims}{{\rm dim}}

\DeclareMathOperator{\diam}{diam}

\newcommand{\rad}{\mathrm{rad}}

\def\vint_#1{\mathchoice
          {\mathop{\vrule width 6pt height 3 pt depth -2.5pt
                  \kern -8pt \intop}\nolimits_{#1}}%
          {\mathop{\vrule width 5pt height 3 pt depth -2.6pt
                  \kern -6pt \intop}\nolimits_{#1}}%
          {\mathop{\vrule width 5pt height 3 pt depth -2.6pt
                  \kern -6pt \intop}\nolimits_{#1}}%
          {\mathop{\vrule width 5pt height 3 pt depth -2.6pt
                  \kern -6pt \intop}\nolimits_{#1}}}

\theoremstyle{plain}
\newtheorem{theorem}[equation]{Theorem}
\newtheorem{corollary}[equation]{Corollary}
\newtheorem{lemma}[equation]{Lemma}

\newtheorem{proposition}[equation]{Proposition}

\theoremstyle{definition}
\newtheorem{definition}[equation]{Definition}

\newtheorem{example}[equation]{Example}
\newtheorem{remark}[equation]{Remark}

 \numberwithin{equation}{section}

 \newcommand{\bi}{\begin{itemize}}
\newcommand{\ei}{\end{itemize}}

\title[Equality of different definitions of conformal dimension]{Equality of different definitions of conformal dimension for quasiself-similar and CLP spaces}
\author{Sylvester Eriksson-Bique}
\address{Department of Mathematics and Statistics, University of Jyv\"askyl\"a,
Seminaarinkatu 15, PO Box 35,  FI-40014 University of Jyv\"askyl\"a, Finland}
\email{sylvester.d.eriksson-bique@jyu.fi}

\thanks{The author was partially supported by Finnish Academy Grants n.~345005 and n.~356861. We thank Mathav Murugan for posing the question to us, for discussing the problem at the Okinawa Institute of Science and Technology in June 2023 and for giving helpful comments on a preprint version of this paper. The work was started at the workshop ``Random walks and analysis on metric spaces''. We thank the institute for its hospitality and care - especially given the weak typhoon that overlapped the event.}

\begin{document}
\maketitle

\begin{abstract}
We prove that for a quasiself-similar and arcwise connected compact metric space all three known versions of the conformal dimension coincide: the conformal Hausdorff dimension, conformal Assouad dimension and Ahlfors regular conformal dimension. This answers a question posed by Mathav Murugan. Quasisimilar spaces include all approximately self-similar spaces. As an example, the standard Sierpi\'nski carpet is quasiself-similar and thus the three notions of conformal dimension coincide for it.  

We also give the equality of the three dimensions for combinatorially $p$-Loewner (CLP) spaces. Both proofs involve using a new notion of combinatorial modulus, which lies between two notions of modulus that have appeared in the literature. The first of these is the modulus studied by Pansu and Tyson, which uses a Carath\'eodory construction. The second is the one used by Keith and Laakso (and later modified and used by Bourdon, Kleiner, Carrasco-Piaggio, Murugan and Shanmugalingam). By combining these approaches, we gain the flexibility of giving upper bounds for the new modulus from the Pansu-Tyson approach, and the ability of getting lower bounds using the Keith-Laakso approach. Additionally the new modulus can be iterated in self-similar spaces, which is a crucial, and novel, step in our argument.
\end{abstract}


\medskip

\tableofcontents

\section{Introduction}

In this paper we study the equivalence of three different versions of the definition of the conformal dimension. First, recall the definition of a quasisymmetry.

\begin{definition}\label{def:quasisym} We say that a homeomoprhic map $f:X\to Y$ is a quasisymmetry, if there exists a homeomorphism $\eta:[0,\infty)\to[0,\infty)$ so that for all $x,y,z\in X$ with $x\neq z$, we have
\begin{equation}\label{eq:quasisym}
\frac{d(f(x),f(y))}{d(f(x),f(z))}\leq \eta\left(\frac{d(x,y)}{d(x,z)}\right).
\end{equation}
We write $X\sim_{q.s.}Y$, if there exists a quasisymmetry $f:X\to Y$. We say that $f$ is an $\eta$-quasisymmetry, if it satisfies \eqref{eq:quasisym} with this specific function $\eta$.
\end{definition}

We consider the effect of quasisymmetries on the a) Hausdorff dimension, b) Assouad dimension and c) Ahlfors-regularity of the space. We briefly recall the definitions of these. If $s\in [0,\infty)$, we define the Hausdorff $s$-content (at scale $\delta\in(0,\infty]$) of $A\subset X$ as
\begin{equation}\label{def:Hausdorffcontent}
\cH^s_\delta(A)=\inf\left\{\sum_{i\in \N} \diam(A_i)^s : A\subset \bigcup_{i\in \N} A_i, \diam(A_i)\leq \delta \right\}.
\end{equation}
For future reference, also define Hausdorff measure by
\[
\cH^s(A):=\lim_{\delta\to 0} \cH^s_\delta(A).
\]
The Hausdorff dimension of $X$ can be defined as
\[
\dims_H(X)=\inf \{s>0: \cH^s_\infty(X)=0\}.
\]
Hausdorff dimension is not stable when sequences of spaces ``converge'' (e.g. in the Gromov-Hausdorff sense), and this is one reason to introduce Assouad dimension. If $A\subset X$ is a subset, let 
\begin{equation}\label{eq:coveringnumber}
N(A,r)=\inf \{N : \exists x_1,\dots, x_N, A \subset \bigcup B(x_i,r)\}.
\end{equation}
Define the Assouad dimension in terms of the scale-invariant asymptotic behaviour of this quantity.
\[
\dims_A(X)= \inf \{s>0 : \exists C>0, \forall R>r>0, \forall z\in X, N(B(z,R),r)\leq CR^sr^{-s}.\}
\]

Finally, a special setting, where $\dims_A(X)=\dims_H(X)$ is when $X$ is Ahlfors regular for some $Q>0$. We say that $X$ is $Q$-Ahlfors regular, if there exists a Radon measure $\mu$ on $X$ and some constant $C\geq 1$, with
\[
C^{-1}r^Q \leq \mu(B(z,r))\leq C r^Q.
\]
for every $z\in X$ and $r\in (0,\diam(X))$. In this case, $Q=\dims_A(X)=\dims_H(X)$. Even if the space is not Ahlfors regular, we always have the inequality $\dims_H(X)\leq \dims_A(X)$, where the inequality may be strict. Further, while not every space is $Q$-Ahlfors regular, spaces often can be deformed into such. A more detailed discussion on this and the three notions of dimension, as well as conformal dimension in general, is given in \cite[Chapter 2]{MTbook}.

 An example of a quasisymmetric map is the identity map $(X,d)\to (X,d^\theta)$, for $\theta\in (0,1)$. Such a map increases all of the three notions of dimension. It is considerably harder to decrease dimension, and one is led to defining three quasisymmetric invariants.
\begin{align*}
\text{Conformal Hausdorff dimension} \quad & \dims_{CH}(X)=\inf\{\dims_H(Y) : X\sim_{q.s.} Y\} \\
\text{Conformal Assouad dimension} \quad & \dims_{CA}(X)=\inf\{\dims_A(Y) : X\sim_{q.s.} Y\} \\
\text{Ahlfors regular conformal dimension} \quad & \dims_{CAR}(X)=\inf\{Q=\dims_H(Y) : X\sim_{q.s.} Y, Y \text{ is  $Q$-Ahlfors regular} \}.
\end{align*}
The first of these was defined in \cite{pansu}, while the final one was used in \cite{bourdonpajot}. The Assouad variant appeared already in \cite{keithlaakso}. 

In general, $\dims_{CH}(X)\leq \dims_{CA}(X)\leq \dims_{CAR}(X)$. For uniformly perfect spaces $\dims_{CAR}(X)=\dims_{CA}(X)$, see \cite[Proposition 2.2.6.]{MTbook} and \cite[Chapters 14 and 15]{hei01}. The relationship between $\dims_{CH}(X)$ and $\dims_{CA}(X)$ has so far not been studied in detail, beyond giving simple examples such as the following, when they are not equal.

\begin{example} Let $X=\Z\times \R$. The conformal Assouad dimension can only drop under blowing the space down, and thus $\dims_{CA}(X)\geq \dims_{CA}(\R^2)=2$. The latter follows since the topological dimension of the plane is $2$, and the Hausdorff dimension is always greater than the topological dimension. However, $\dims_{CH}(X)= \dims_H(X)=1$.

If we set $X=\Z\times \R \cup \R\times \Z$, we can even make $X$ connected without altering the previous argument. It is possible to make the space compact and connected as well: Let $X=(\{\frac{1}{n} : n\in\N\}\cup\{0\} )\times [0,1] \cup  [0,1] \times (\{\frac{1}{n}: n\in\N\}\cup\{0\} )$. In this case, a blow-up of the space is $\R^2$.
\end{example}

Assouad dimension involves a scale-invariant quantitative condition, while Hausdorff dimension is merely a qualitative statement on the dimension of the space. 
Further, as the previous example indicates $\dims_{CA}(X)$ has stability properties under limits, while $\dims_{CH}(X)$ does not. This means, that one may only hope for their equality in the case where one assumes some form of self-similarity. Consequently Mathav Murugan asked if the different definitions of conformal dimension agree for self-similar spaces \cite{murugan2022conformal}. Our main theorem answers this intuition in the affirmative. The notion of quasiself-similarity is given in Definition \ref{def:appxquasi}, and (to our knowledge) was introduced in \cite{piaggio2011jauge}.

\begin{theorem}\label{thm:mainthm}
Let $X$ be a compact quasiself-similar metric space, which is connected and locally connected. Then, 
\[
\dims_{CH}(X)=\dims_{CA}(X)=\dims_{CAR}(X).
\]
\end{theorem}
As stated, the equality $\dims_{CA}(X)=\dims_{CAR}(X)$ for uniformly perfect spaces was already known, and follows directly from
\cite[Proposition 2.2.6]{MTbook}. Our contribution is to prove $\dims_{CH}(X)=\dims_{CA}(X)$. Indeed, this equality has many further consequences. One may define a zoo of other conformal dimensions, such as: conformal upper and lower Minkowski dimension, conformal packing dimension... Since these dimensions lie between the Hausdorff dimension and the Assouad dimension, one gets equality for the corresponding notions of conformal dimension as well.

The only other result, which states equality of $\dims_{CH}(X)$ with $\dims_{CAR}(X)=\dims_{CA}(X)$ is that of \cite[Theorem 3.4]{tysonconfdim} and \cite[Proposition 2.9.]{pansu}, which apply when $X$ is $Q$-Ahlfors regular and possesses a curve family with positive continuous $Q$-modulus. We will discuss this further below. We are not aware of any other instances, where equality of all notions has been shown.
 
A concrete corollary of Theorem \ref{thm:mainthm} is the following new result. The $n$-dimensional Sierpi\`nski sponge $M_n$ is obtained by iteratively subdividing the side of an $n$-dimensional cube by three, and removing the central cube.

\begin{corollary}\label{cor:sierpinski}Let $n\geq 2$. If $M_n$ is an $n$-dimensional Sierpsi\`nski sponge, then 
\[
\dims_{CH}(M_n)=\dims_{CA}(M_n)=\dims_{CAR}(M_n).
\]
\end{corollary}

We will also give a result for non-self-similar spaces, where self-similarity is replaced with the combinatorial Loewner property (CLP) from \cite{bourdonkleiner, clais}; see Section \ref{sec:combloew} for a definition. It is worth noting, that this assumption usually is verified in the self-similar setting, and thus is not so much more general than Theorem \ref{thm:mainthm}. We present this here, since the argument for it is a bit simpler than for the general self-similar case.  Further, it is worth to record a proof for this result here, since the developed tools may be useful in tackling the question of Bruce Kleiner, which asks if self-similar combinatorially Loewner spaces are quasisymmetric to Loewner spaces; see \cite{kleinericm} for further background and the question. 
\begin{theorem}\label{thm:mainthmcombloew}Let $p\in (1,\infty)$.
Let $X$ be a compact, doubling and LLC space, which is $p$-combinatorially Loewner metric space. We have
\[
\dims_{CH}(X)=\dims_{CA}(X)=\dims_{CAR}(X)=p.
\]
\end{theorem}
We note that Corollary \ref{cor:sierpinski} would also follow from this result, since Sierpi\`nski sponges are $p$-combinatorially Lowener spaces; see the proofs in \cite{bourdonkleiner}. In the course of the proof of Theorem \ref{thm:mainthmcombloew} we will in fact present some stronger results for CLP spaces in Section \ref{sec:combloew}. In fact, while the statement $\dims_{CH}(X)=p$ is qualitative, we will give a quantitative statement, Proposition \ref{prop:Hausdimquant}, which gives a lower bound for the Hausdorff measure of the images of balls under quasisymmetries. This inequality may be useful in other settings as well, and is a generalization of an inequality which appeared in the work of Heinonen and Koskela \cite[Theorem 3.6]{heinonenkoskela}.

Our results here clarify a central point of ambiguity in much of the literature on conformal dimension, where the equality of the conformal Hausdorff and conformal Assouad dimensions is not addressed, but rather avoided and bypassed. We next describe the main idea of the proof.

The key tool in a majority of the research on conformal dimension is a notion of modulus - in particular discretized versions of moduli of path families. These generalize the notion of continuous modulus (later, often, modulus), see e.g. \cite{fuglede, hei01,shabook} for background. Our proof is also based on defining a new type of discrete modulus - or rather, discrete admissibility - and relating it to conformal Hausdorff dimension. At this point, there are several variants of discrete modulus, each with its own setting and application see e.g. \cite{pansu,tysonconfdim,keithlaakso, carrasco, shanmunconf, murugan2022conformal,jeffmod,abppcw:ecgd2015}. (There are also other notions, such as trans-boundary modulus, see e.g. \cite{schramm,bonkmerenkov}, but these are not relevant for our discussion here.)  We will not discuss all these moduli here, but will focus on those which motivate our approach.

The motivation for our argument and notion of modulus comes from a result of Pansu \cite[Proposition 2.9.]{pansu}, whose dual formulation\footnote{As a side note, we remark that Pansu considers measures on families of curves, while Tyson uses the notion of curve modulus in \cite{tysonconfdim}. These two notions are roughly dual to each other, see e.g. \cite{dualitymodulus, david2020infinitesimal} for more precise statements.} was given by Tyson in \cite[Theorem 3.4]{tysonconfdim}. Tyson shows that if $X$ is a $Q$-Ahlfors regular metric measure space, and if it possesses a family of curves $\Gamma$ with positive continuous modulus, then $\dims_{CH}(X)=Q$. The proof of Tyson uses the discrete $Q$-modulus as introduced by Pansu in \cite{pansu}. To be very brief, this modulus is defined using a Carath\`eodory construction and involves discrete sums. One shows, both in \cite{pansu} and \cite{tysonconfdim}, that the discrete $Q$-modulus is bounded from below by the continuous $Q$-modulus. Further, the discrete modulus, up to a variation of parameters, is invariant under quasisymmetries. The final nail in the coffin of the proof is that if $\dims_H(Y)<Q$, then the discrete $Q$-modulus vanishes on $Y$. Consequently, a family of curves with positive continuous $Q$-modulus  obstructs lowering the dimension of $X$ by a quasisymmetry below $Q$.  

The previous proof relies heavily on the fact that we can use the notion of continuous modulus to give a lower bound for discrete modulus. In many settings, such as the Sierpi\'nski sponges mentioned above, the continuous moduli of all curves vanishes. Thus, we lack this lower bound, and we need find a way around this by giving a lower bound using a different quantity. In the quasiself-similar setting, and in the combinatorially Loewner setting, we can obtain this lower bound by slightly different mechanisms - and by employing a different modulus.

In the work \cite{keithlaakso, carrasco,murugan2022conformal}, the inability to lower the dimension can be converted to a lower bound on some moduli - see Theorem \ref{thm:confhausdim} for a precise statement. Thus, one can use their result to obtain a lower bound for a different discrete modulus, which we call the Keith-Laakso modulus and which is defined in subsection \ref{subsec:discretemoduli}.  In the case of combinatorially Loewner spaces, the setting is a bit simpler and the lower bound is obtained directly by the assumption that the space is combinatorially Loewner  \cite{bourdonkleiner,clais}: see Definition \ref{def:combloew}. 

At this juncture, we have two moduli: the Keith-Laakso modulus and that of Pansu and Tyson. For the first we can obtain lower bounds. For the second, one can show upper bounds. Indeed, for Pansu and Tyson, the notion of discrete modulus is such that it is very easy to prove that if $\dims_H(Y)<Q$, the discrete modulus vanishes. In the absence of $Q$-Ahlfors regularity, it is harder to give lower bounds for the modulus of Pansu and Tyson. On the other hand, for the Keith-Laakso modulus, one lacks the ability to give good upper bounds and thus to directly say that the discrete modulus vanishes if one has Hausdorff dimension lower than $Q$.

The reason for this inability is the following technical, but crucial point. The definition of Keith-Laakso modulus can be summarized as assigning a value $\Mod_p^{KL}(\Gamma, \cU)$ for a specific curve family $\Gamma$ and a cover $\cU$ of $X$, which \cite{bourdonkleiner} calls a $\kappa$-approximation \emph{at some level $r$}. The key feature of their $\kappa$-approximations is that all sets in $\cU$ have roughly the same size. (See Subsection \ref{subsec:kappaapprox} for details.) This is also a key difference with the work in \cite{pansu,tysonconfdim}, since there the Carath\'eodory construction involves \emph{arbitrary covers}. 

Similarly, Assouad dimension involves covering the space by balls of the same size, whereas Hausdorff dimension involves coverings by sets of various sizes. To give estimates for Hausdorff dimension, we need to allow arbitrary covers in the definition of discrete modulus. We bridge this gap, by introducing a new notion of modulus $\overline{\Mod}_p(\Gamma,\cU)$ which lies between those of Pansu and Tyson in \cite{pansu, tysonconfdim} and Keith, Laakso and others in \cite{keithlaakso, carrasco,murugan2022conformal}. First, we get more flexibility by allowing arbitrary covers $\cU$ that consist of balls (or, in general, sufficiently round sets). This forces us to introduce a new admissibility condition, to address several key technical issues. Similar to Pansu's discrete modulus, we can show that if $\dims_H(Y)<Q$, then this modulus is very small for a given cover. Further, in the self-similar and CLP space settings, we can relate the Keith-Laakso modulus and the new modulus to each other.

For combinatorially Loewner spaces, the story is easier to finish. One can bound $\overline{\Mod}_p(\Gamma,\cU)$ from below using the Keith-Laakso modulus, which in term has a lower bound from the combinatorial Loewner assumption. This estimate is given in Proposition \ref{prop:discretemoduluscombloew}. This gives a contradiction to the previous paragraph's conclusion of $\overline{\Mod}_p(\Gamma,\cU)$ being small. In fact, this argument is somewhat easier to discover,  and it served as a starting point for this paper and project. For this reason we also include the argument in this paper. Quickly, however, the author realized that a more technical version of the argument could be applied for general quasiself-similar spaces.

For \emph{quasiself-similar spaces} the argument is a bit different. Instead of directly using a lower bound, we use the fact that the ability to lower dimension gives an upper bound. Indeed, if there is a quasisymmetric map $f:X\to Y$ and if $Y$ has small Hausdorff measure, then we obtain a quantitative statement on moduli of annuli, see Lemma \ref{lem:Hausdorff} and Proposition \ref{prop:smallmodulus}. Our quantitative statement can be converted algorithmically by using iteration to a statement on the smallness of the Keith-Laakso modulus.  This allows us to prove the equality $\dims_{CA}(X)=\dims_{CH}(X)$ for quasiself-similar spaces by using the result of Carrasco-Paggio, which we state below in Theorem \ref{thm:confhausdim}. The iteration is algorithmic, but quite technical. The basic step of the iteration involves ideas from the proof of the result for CLP spaces. We will describe it in more detail in Subsection \ref{subsec:algorithm}.


\subsection{Outline}

We will present some general terminology in Section \ref{sec:notation}. Then, in Section \ref{sec:discretemoduli} we introduce the different notions of discrete modulus needed in this paper, and present some known results on their relationships with the conformal dimension. For technical reasons, we will use mostly a variant of this modulus, the Bourdon-Kleiner modulus defined in \cite{bourdonkleiner}, instead of the Keith-Laakso modulus. However, we will relate the two moduli to each other. In Subsection \ref{subsec:newmod}, we give the new modulus that is key to the approach of this paper. In Section \ref{sec:combloew} we focus on CLP spaces. There, we prove Theorem \ref{thm:mainthmcombloew}, which is the equality of the definitions of conformal dimension for CLP spaces. In the process, we give some useful stronger results on discrete moduli, and precise quantitative estimates, which hold for CLP spaces. In Section \ref{sec:quasiselfsim} we focus on quasiself-similar spaces. There, we study moduli of annuli, and give an push-down algorithm to adjust the scale of covers. This is then used to give a relationship between the two moduli used. Finally, in subsection \ref{subsec:proofmain} we collect the pieces and complete the proof of Theorem \ref{thm:mainthm}.

\section{Notation and Basic properties}\label{sec:notation}

\subsection{Basic terminology}

A compact metric space will be denoted $X$, its metric $d$, and open balls within it $B(z,r):=\{w\in X: d(z,w)<r\}$ for $z\in X, r>0$. An inflation of a ball $B=B(z,r)$ is denoted $CB:=B(z,Cr)$ for $C>0$. Note that we consider each ball as having an associated center and radius -- and it may happen that a different center and radius defines the same set. The radius of a ball is denoted $\rad(B)$. Diameters of sets $A\subset X$ will be denoted $\diam(A)=\sup_{a,b\in A} d(a,b)$. A curve is a continuous map $\gamma:I\to X$, where $X$ is a non-empty compact interval in $\R$. We often conflate $\gamma$ and it's image set $\image(\gamma)$.

Recall the definition of $N(A,r)$ from \eqref{eq:coveringnumber}. We say that a metric space $X$ is metrically doubling, if there exists a constant $D\geq 1$, so that $N(B(z,r),r/2)\leq D$ for every $z\in X$ and $r>0$. 

We will need some connectivity properties. A space $X$ is called locally connected, if it has a neighborhood basis consising of connected open sets. A metric space is LLC, if for every $x,y\in X$, there exists a curve $\gamma$ with $x,y\in \gamma$ and $\diam(\gamma)\leq C d(x,y)$.

We will consider collections of balls, which are often denoted by a script letter $\cB$. For these, we define unions by setting $\bigcup \cB:=\bigcup_{B\in \cB} B$, inflations by setting $C\cB:=\{CB: B\in \cB\}$ and radii $\rad(\cB)=\sup_{B\in \cB} \rad(B)$. If $A$ is any finite set, we denote by $|A|$ its cardinality.

\subsection{Relative distance and quasisymmetries}

We need some standard results on quasisymmetries.
\begin{lemma}\label{lem:quasisyminv} If $f:X\to Y$ is an $\eta$-quasisymmetry, then $f^{-1}$ is a $\tilde{\eta}$-quasisymmetry with $\tilde{\eta}(t)=\left(\eta^{-1}(t^{-1})\right)^{-1}$.
\end{lemma}
We note the convention that the value of $\tilde{\eta}$ at zero is given by $\tilde{\eta}(0)=0$. 
\begin{proof}[Proof of Lemma \ref{lem:quasisyminv}]
Let $x,y,z\in Y$ and let $x',y',z'\in X$ be such that $f(x')=x,f(y')=y,f(z')=z$. Since $f$ is an $\eta$-quasisymmetry, we have
\[
\frac{d(f(x'),f(z'))}{d(f(x'),f(y'))}\leq \eta\left(\frac{d(x',z')}{d(x',y')}\right).
\]
Taking resiprocals and an inverse function, we get
\[
\frac{d(x,y)}{d(x,z)}\leq \left(\eta^{-1}\left(\left(\frac{d(f(x'),f(y'))}{d(f(x'),f(z'))}\right)^{-1}\right)\right)^{-1}.
\]
Replacing $f(x'),f(y'),f(z')$ with $x,y,z$ and $x',y',z'$ with $f^{-1}(x),f^{-1}(y),f^{-1}(z)$ yields that $f^{-1}$ is an $\tilde{\eta}$-quasisymmetry.
\end{proof}

Let $X$ be a complete metric space. A continuum $E\subset X$ is a compact connected set. A continuum is non-degenerate, if it is non-empty. We define the relative distance between two non-degenerate continua $E,F$ as
\[
\Delta(E,F):=\frac{d(E,F)}{\min\{\diam(E),\diam(F)\}}.
\]

\begin{lemma} \label{lem:relativedist} Let $f:X\to Y$ be an $\eta$-quasisymmetry and let $E,F$ be two non-degenerate disjoint continua in $X$. Then,
\[
\frac{1}{2\eta(\Delta(E,F)^{-1})}\leq \Delta(f(E),f(F))\leq \eta(2\Delta(E,F)).
\]
\end{lemma}
\begin{proof}
Assume by symmetry that $\diam(E)\leq \diam(F)$. Let $x\in E$ and $y\in F$ be such that $d(E,F)=d(x,y)$. Choose $u\in E, v \in F$ so that $d(x,u),d(y,v)\geq \diam(E)/2$. This is possible by connectivity. Then, we have
\[
\frac{d(x,y)}{d(x,u)}\leq 2\Delta(E,F) \quad \text{ and } \frac{d(y,x)}{d(y,v)}\leq 2\Delta(E,F)
\]

Let $x':=f(x),y'=f(y),u'=f(u), v'=f(v)$ be the image points in $Y$. We have, since $\eta$ is increasing and since $f$ is an $\eta$-quasisymmetry:
\begin{align*}
d(f(E),f(F))&\leq d(x',y') \\
&\leq \eta\left(\frac{d(x,y)}{d(x,u)}\right)d(x',u')\\
&\leq \eta\left(2\Delta(E,F)\right)\diam(f(E)).
\end{align*}
Similarly, 
\begin{align*}
d(f(E),f(F))&\leq d(y',x') \\
&\leq \eta\left(\frac{d(y,x)}{d(y,v)}\right)d(y',v')\\
&\leq \eta\left(2\Delta(E,F)\right)\diam(f(F)).
\end{align*}
The previous two inequalities combine to gives the inequality:
\[
\Delta(f(E),f(F))= \frac{d(f(E),f(F))}{\min\{\diam(f(E)),\diam(f(F))\}} \leq \eta\left(2\Delta(E,F)\right).
\]
Applying this to the inverse $f^{-1}$, which by Lemma \ref{lem:quasisyminv} is an $\tilde{\eta}$-quasisymmetric map, yields  the other inequality of the claim. 
\end{proof}

The following lemma will also prove useful on a few occasions. Note that the additional assumption on the existence on $y\in B(x,r)$ is automatically satisfied if $X$ is connected and $r<\diam(X)$.

\begin{lemma}\label{lem:ballinclusion} Let $f:X\to Y$ be a quasisymmetric map and let $B(x,r)$ be a ball in $X$ for which there exists a $y\in B(x,r)$ with $d(x,y)\geq r/2$. Then, for every $L\geq 1$, we have
\[
f(B(x,Lr))\subset B(f(x),\eta(2L)d(f(x),f(y))).
\]

\end{lemma}
\begin{proof}
Let $z\in B(x,Lr)$, and apply the $\eta$-quasisymmetry to the triple of points $x,y,z$. This gives
\[
\frac{d(f(x),f(z))}{d(f(x),f(y))}\leq \eta\left(\frac{d(x,z)}{d(x,y)}\right)\leq \eta(2L).
\]
Consequently, we get the claim from
\[
d(f(x),f(z))\leq \eta(2L)d(f(x),f(y)).
\]
\end{proof}

\subsection{Quasiself-similarity}

We define a notion of quasiself-similarity. This is motivated by the notion of approximate self-similarity discussed in \cite{bourdonkleiner}.

\begin{definition}\label{def:appxquasi} We say that a compact space $X$ is quasiself-similar, if there exists a homeomorphism $\eta:[0,\infty)\to[0,\infty)$ and a constant $\delta>0$ so that for $B(x,r)\subset X$ there is a $\eta$-quasisymmetry $f:B(x,r)\to U_{x,r}$ where $U_{x,r}\subset X$ is an open set with $\diam(U_{x,r})\geq \delta \diam(X)$.   We also say that $X$ is $\eta$-quasiself-similar, if this property holds for a given function $\eta$.
\end{definition}

The principal advantage of defining quasiself-similar spaces is that they are more general than approximately self-similar spaces. Further, quasiself-similarity is an invariant under quasisymmetries: if $X$ is quasiself-similar and $Y\sim_{q.s.} X$, then $Y$ is also quasiself-similar. The same fails for approximate self-similarity. 

We recall the following result \cite[Chapter 2]{piaggio2011jauge}.

\begin{lemma}\label{lem:quasiselfconn} If $X$ is a compact quasiself-similar space, which is connected and locally connected, then $X$ is LLC.
\end{lemma}

\subsection{$\kappa$-approximations}\label{subsec:kappaapprox}

We introduce some terminology on approximations. Throughout this paper, $\cU$ and $\cV$ will denote finite collections of open sets.

\begin{definition}\label{def:kround} Let $\kappa \geq 1$. A finite collection of open sets $\cU$ of a metric space $X$ is called a $\kappa$-round collection, if for every $U\in \cU$ there exists a $z_U$ so that
\[
B(z_U,\kappa^{-1}r_U)\subset U \subset B(z_U, r_U),
\]
where $r_U=\sup\{d(z_U,x): x\in U\}$. If further there is some $r>0$, so that $r_U=r$ for every $U\in \cU$, we call $\cU$ a $\kappa$-round collection at level $r$.
\end{definition}

From here on out, if $U$ is any open set and $z_U\in U$ has been fixed, we define $r_U:=\sup\{d(z_U,x): x\in U\}$.

\begin{definition}\label{def:kapprox} Let $\kappa \geq 1$. A $\kappa$-round collection of open sets $\cU$ of a metric space $X$ is called a $\kappa$-locally bounded collection, if there exist $z_U\in U$ for every $U\in \cU$ for which Definition \ref{def:kround} holds and for which moreover the following two properties hold.
\begin{enumerate}
\item The balls $\{B(z_U,\kappa^{-1}r_U), U \in \cU\}$  are pairwise disjoint.
\item For every $L\geq 1$, there exists a constant $\kappa_L$ so that if $B(z_U,L r_U)\cap B(z_V,L r_V) \neq\emptyset$, then $r_U\leq \kappa_L r_V$.
\end{enumerate} 

If $\cU$ also covers $X$, then we call it a $\kappa$-approximation.
If further there is some $r>0$, so that $r_U=r$ for every $U\in \cU$, we call $\cU$ a $\kappa$-approximation at level $r$.
\end{definition}

Let $\rad(\cU)=\sup\{r_U : U \in \cU\}$. A standard way to obtain a $\kappa$-approximation is the following. A set $N\subset X$ is called $r$-separated if for all $x,y\in X$ we have $d(x,y)\geq r$. A maximal $r$-separated set is called an $r$-net. Given any $r$-net $N$ in a connected space $X$, with $r\in (0,\diam(X)/2)$, it is straightforward to show that the collection $\cU=\{B(x,2r): x\in N\}$ is a $\kappa$-approximation at level $r$ with $r_U=2r$ and $z_U=x$ for every $U=B(x,2r)\in \cU$, and $\kappa=1,\kappa_L=4$ for all $L\geq 1$. 

We note that we have made some adjustments in the notation and terminology to bridge small differences in the literature, and in order to connect more directly to our work. The following remark explains some of these choices and how the other definitions/concepts can be expressed in our framework.

\begin{remark}\label{rmk:defapprox} We briefly explain the relationships between different definitions used in \cite{keithlaakso,carrasco,murugan2022conformal,shanmunconf} and \cite{bourdonkleiner}. In the first four of these, one takes $\alpha\geq 2$ and considers a sequence $N_k$ of $\alpha^{-k}$ nets and a parameter $\lambda> 1$, and defines graphs $G_n$ whose vertex set is $N_k$, and with edges $v,w$ if $B(v,\lambda 2^{-k})\cap B(w,\lambda 2^{-k})\neq \emptyset$. In our case, this would correspond to the $\kappa$-approximation given by $\cU=\{B(v,\lambda 2^{-k})\}$, and setting $\kappa = 2\lambda$. Doing so, the incidence graph associated to $\cU$ is isomorphic to that of $G_n$. This isomorphism is relevant in Section \ref{sec:discretemoduli}, since we will define discrete moduli using incidences of sets in $\cU$, while in  \cite{keithlaakso,carrasco,murugan2022conformal,shanmunconf} the moduli are defined in the graphs $G_n$. Since these two graphs are isomorphic, the relevant notions of moduli coincide.

On the other hand, compared to \cite{bourdonkleiner} we use a slightly  more general framework of arbitrary $\kappa$-approximations in order to ensure the quasisymmetry invariance of our definitions. In their work, one only uses $\kappa$-approximations at a given level $r$.
\end{remark}

Let $\cV$ be a $\kappa$-round collection in $X$, and let $f:X\to Y$ be a quasisymmetry. Then define the image collection $f(\cV):=\{f(V):V\in \cV\}$. We then then have the following.

\begin{lemma}\label{lem:kapparoundinvariance} Let $\cV$ be a $\kappa$-round collection in a space $X$ and if $f:X\to Y$ is an $\eta$-quasisymmetric map, then $f(\cV)$ is a $\kappa'$-round collection with $\kappa'=2\eta(\kappa)$.

Moreover, if $\cV$ is a $\kappa$-approximation, then $f(\cV)$ is a $\kappa'$-approximation with $\kappa'=2\eta(\kappa)$.
\end{lemma}
\begin{proof}
For every $V\in \cV$ let $z_V\in V, r_V>0$ be the center and radius specified in Definition \ref{def:kround}. Define $z_{f(V)}=f(z_V)$ and $r_{f(V)}=\sup\{d(y,f(z_V)) : y\in f(V)\}$. 

Suppose first that $\cV$ is $\kappa$-round and let $\kappa'=2\eta(\kappa)$. We will show that $V$ is $\kappa'$-round, that is, we prove
\begin{equation}\label{eq:inclusions}
B(z_{f(V)},\kappa'^{-1}r_{f(V)})\subset f(V) \subset B(z_{f(V)},r_{f(V)}).
\end{equation} 
The second of these inclusions follows from the definition of $r_{f(V)}$. Now, let $y\in B(z_{f(V)},\kappa'^{-1}r_{f(V)})$, and let $b\in X$ be such that $f(b)=y$. Choose a point $w\in f(V)$ so that $d(w,z_{f(V)})\geq 2^{-1}r_{f(V)}$ and let $c\in V$ be such that $f(c)=w$. Since $f$ is a quasisymmetry, we get
\[
\kappa'/2=\frac{2^{-1}r_{f(V)}}{\kappa'^{-1}r_{f(V)}}\leq \frac{d(z_{f(V)},w)}{d(z_{f(V)},y)}\leq \eta\left(\frac{d(z_V,c)}{d(z_V,b)}\right).
\]
Thus,
\[
d(z_V,b)\leq d(z_V,c) \eta^{-1}(\kappa'/2)^{-1}\leq r_V \kappa^{-1}.
\]
Therefore $b\in B(z_V,r_V\kappa^{-1})\subset V$ and $y\in f(V)$. This yields the first of the inclusions in \eqref{eq:inclusions}. Thus, $f(\cV)$ is $\kappa'$-round.

Let us know assume further that $\cV$ is a $\kappa$-approximation. Indeed, it is $\kappa$-locally bounded and covers $X$. Clearly $f(\cV)$ covers $Y$. Thus, it suffices to prove that $f(\cV)$ is $\kappa'$-locally bounded.

The proof above showed in fact that
\begin{equation}\label{eq:inclusions2}
B(z_{f(V)},\kappa'^{-1}r_{f(V)})\subset f(B(z_{V},\kappa^{-1}r_{f(V)})).
\end{equation}
Thus, the balls $\{B(z_{f(V)},\kappa'^{-1}r_{f(V)}) : V\in \cV\}$ are pairwise disjoint. Therefore, we are left to show that for every $L\geq 1$ there exists a $\kappa_L'$ so that if
\[
B(z_{f(V)}, L r_{f(V)}) \cap B(z_{f(U)}, L r_{f(U)})\neq\emptyset
\]
for some $U,V\in \cV$, then $r_{f(U)}\leq \kappa_L' r_{f(V)}$. This is obtained by first finding an $L'\geq 1$ so that $B(z_{V}, L' r_{V}) \cap B(z_{U}, L r_{U})\neq \emptyset$, which yields an estimate for $d(z_U,z_V)$ in terms of $r_V$, and then using the quasisymmetry to translate this into a bound for $r_{f(U)}$ in terms of $r_{f(V)}$.

Let $w\in B(z_{f(V)}, L r_{f(V)}) \cap B(z_{f(U)}, L r_{f(U)})$ and let $u\in B(z_{f(U)}, r_{f(U)}), v\in B(z_{f(V)}, r_{f(V)})$ be points with $d(u,z_{f(U)})\geq r_{f(U)}/2$ and $d(v,z_{f(V)})\geq r_{f(V)}/2$. Let $a\in X,b_U\in U,b_V\in V$  
points 
so that $f(a)=w, f(b_U)=u, f(b_V)=v$.

By Lemma \ref{lem:quasisyminv}, the map $f^{-1}$ is a $\tilde{\eta}$-quasisymmetry, with $\tilde{\eta}(t)=\left(\eta^{-1}(t^{-1})\right)^{-1}$.  By the quasisymmetry condition applied to the three points $b_U,a,z_U$, we have 
\[
d(w,z_U) \leq \tilde{\eta}\left(\frac{d(w,z_{f(U)})}{d(u,z_{f(U)})}\right)d(z_U,b_U)\leq \tilde{\eta}(2L) r_U.
\]
Thus, $a\in B(z_U,\tilde{\eta}(2L) r_U)$. Similarly, we get 
$a\in B(z_V,\tilde{\eta}(2L) r_V).$
Consequently 
\begin{equation}\label{eq:ainclusion}
a\in B(z_U,\tilde{\eta}(2L)r_U) \cap B(z_V,\tilde{\eta}(2L)r_V).
\end{equation}
Therefore, since $\cU$ is locally bounded, there exists a constant $\kappa_{\tilde{\eta}(2L)}$ for which
\begin{equation}\label{eq:rubound}
\tilde{\eta}(2L)^{-1}r_V \leq r_U \leq \kappa_{\tilde{\eta}(2L)}r_V.
\end{equation}

From \eqref{eq:ainclusion} and \eqref{eq:rubound} we get 
\[d(z_U,z_V) \leq d(z_U,a)+d(z_V,a)\leq \tilde{\eta}(2L)(1+\kappa_{\tilde{\eta}(2L)})r_V.
\] 
We have $z_U \in B(z_V,\tilde{\eta}(2L)(1+\kappa_{\tilde{\eta}(2L)})r_V)$. Again, by Lemma \ref{lem:ballinclusion}, we get that
\[
z_{f(U)}=f(z_U)\in B(z_{f(V)}, \eta(2\tilde{\eta}(2L)(1+\kappa_{\tilde{\eta}(2L)}))r_{f(V)}).
\]
In particular, 
\begin{equation}\label{eq:zuzvbound}
d(z_{f(U)}, z_{f(V)})\leq \eta(2 \tilde{\eta}(2L)(1+\kappa_{\tilde{\eta}(2L)}))r_{f(V)}.
\end{equation}
We also have $z_{V}\not\in B(z_{U}, \kappa^{-1}r_{U})$, and thus
\begin{equation}\label{eq:zuzvlowerbound}
d(z_{U}, z_{V})\geq \kappa^{-1}r_{U}.
\end{equation}

Finally, apply the $\eta$-quasisymmetry to the points $z_U,z_V$ and $b_U$ and use  \eqref{eq:zuzvlowerbound} to give
\begin{equation}\label{eq:qsrfUbound}
\frac{r_{f(U)}}{2 d(z_{f(V)}, z_{f(U)})} \leq \frac{d(z_{f(U)}, u)}{d(z_{f(V)}, z_{f(U)})} \leq \eta\left(\frac{d(z_{U}, b_{U})}{d(z_{U}, z_V)}\right) \leq \eta\left(\frac{\kappa r_U}{r_U}\right)\leq \eta(\kappa).
\end{equation}
Thus, by applying \eqref{eq:zuzvbound} we get
\[
r_{f(U)} \leq 2 \eta(\kappa) \eta(2 \tilde{\eta}(2L)(1+\kappa_{\tilde{\eta}(2L)}))r_{f(V)}.
\]
This is the desired estimate with $\kappa_L'=2 \eta(\kappa) \eta(2 \tilde{\eta}(2L)(1+\kappa_{\tilde{\eta}(2L)}))$ and yields the local boundedness.



\end{proof}

\section{Discrete moduli}\label{sec:discretemoduli}

\subsection{Discrete modulus of a collection} \label{subsec:discretemoduli}

We will define all the relevant discrete moduli in this section. First, we define a discrete modulus of a collection of discrete subsets. Let $\cU$ be a $\kappa$ round collection and let $\cP$ be a collection of subsets of $\cU$. (Indeed, in general $\cU$ could be any finite collection of objects, but in our application, we will restrict to such collections.) We say that $\rho:\cU\to [0,\infty)$ is discretely admissible for $\cP$, and write $\rho \wedge_{\cU} \cP$, if
\[
\sum_{U\in \cP} \rho(U) \geq 1, \text{ for all } P\in \cP.
\]

Define the discrete modulus by
\[
\Mod_p^D(\cP,\cU):=\inf_{\rho \wedge_{\cU} \cP} \sum_{U\in \cU} \rho(U)^p.
\]
The sum on the right will often also be called the \emph{$p$-energy} of $\rho$.

We recall some basic properties of modulus, whose proofs are standard. For similar arguments, see e.g. \cite[Section 1]{fuglede}. The existence of minimizers follows directly from the fact that $\cU$ must be finite, and the optimization is done in a finite dimensional space. 

\begin{lemma}\label{lem:modulus} Let $\cU$ be a $\kappa$-round collection of $X$ and let $p\geq 1$.
\begin{enumerate}
\item Monotonicity: If $\cP\subset \cP'$ are two collections of sets, then 
\[
\Mod_p^D(\cP,\cU) \leq \Mod_p(\cP',\cU).
\]
\item Sub-additivity: If $\cP,\cP'$ are two collections of subsets, then 
\[
\Mod_p^D(\cP\cup \cP',\cU) \leq \Mod^D_p(\cP',\cU)+\Mod^D_p(\cP,\cU).
\]
\item Majorization: If $\cP,\cP'$ are two collections of subsets so that every set $P\in \cP$ contains a subset in $\cP'$, then 
\[
\Mod_p^D(\cP,\cU) \leq \Mod^D_p(\cP',\cU).
\]
\item Existence of minimizers: If $X$ is compact, then there exists a $\rho \wedge_{\cU} \cP$ with 
\[\Mod^D_p(\cP,\cU)= \sum_{U\in \cU} \rho(U)^p.\]
\end{enumerate}
\end{lemma}

In what follows, since these properties are so standard, we will often simply apply these facts without explicit reference to this Lemma.

\subsection{Modulus of annulus}\label{subsec:moduli}

Let $B$ be a ball in $X$ and $L>1$. Consider a $\kappa$-round collection $\cU$. We say that $P=\{U_1,\dots, U_n\}\subset \cU$ is a $(\cU,B,L)$-path, if $U_1\cap B\neq \emptyset$, $U_n \cap X\setminus LB\neq \emptyset$ and if $U_i\cap U_{i+1} \neq \emptyset$ for all $i=1,\dots, n-1$. Let $\cP_{\cU,B,L}$ be the collection of all $(\cU,B,L)$-paths.

Then, we define the Keith-Laakso modulus as
\[
\Mod_{p,L,\cU}^{KL}(B):=\Mod_p^D(\cP_{\cU,B,L},\cU) .
\]
This notion of modulus coincides with that of \cite{keithlaakso,carrasco,murugan2022conformal} if we use the collection $\cU$ indicated in Remark \ref{rmk:defapprox}.  

In \cite{bourdonkleiner}, a slightly different form of modulus is obtained by using a more restrictive admissibility constraint. Note that this modulus is defined for collections of curves, while the previous one is only defined for balls (and corresponds to a family of objects which traverse an annulus). Let $\Gamma$ be a family of curves, and let $\cP_\Gamma = \{P_\gamma : \gamma \in \Gamma\}$, where $P_\gamma =\{U\in \cU : U\cap \gamma \neq \emptyset\}$. We define the (Bourdon-Kleiner) modulus of the curve family as 
\[
\Mod_{p,\cU}(\Gamma)=\Mod_p^D(\cP_\Gamma, \cU).
\]

We relate the two moduli via the following lemma. This justifies using the Bourdon-Kleiner modulus, instead of the Keith-Laakso modulus in the context of conformal dimension, see Theorem \ref{thm:confhausdim}. Let $\Gamma_{L,B}$ be the collection of curves $\gamma$ connecting $B$ to $X\setminus LB$.

\begin{lemma}\label{lem:moduli} Suppose that $X$ is compact, metrically doubling and LLC. Then, for any $\kappa$-approximation $\cU$, any ball $B\subset X$ and any $L>1$, we have
\[
\Mod_{p,\cU}(\Gamma_{B,L})\sim \Mod_{p,L,\cU}^{KL}(B).
\]
\end{lemma}
\begin{proof}
We have $\cP_\Gamma \subset \cP_{\cU,B,L}$, where $\Gamma=\Gamma_{B,L}$. From the definition of the modulus, and Lemma \ref{lem:modulus} it is thus direct that
\[
\Mod_{p,\cU}(\Gamma_{B,L})=\Mod^D_p(\cP_\Gamma, \cU)\leq \Mod^D_p(\cP_{\cU,B,L},\cU))=\Mod_{p,L,\cU}^{KL}(B).
\]

For the other direction of the proof we need the assumptions of $X$ being LLC and metrically doubling. Let $X$ be $C$-LLC and $D$-metrically doubling. Next, let $\rho \wedge_{\cU} \cP_{\Gamma}$ be arbitrary. We will define another $ \tilde{\rho}$ so that $\tilde{\rho} \wedge_{\cU} \cP_{\cU,B,L}$ and so that 
\begin{equation}\label{eq:tilderhomassbnd}
\sum_{U\in \cU} \tilde{\rho}(U)^p \leq M \sum_{U\in \cU} \rho(U)^p
\end{equation}
for a constant $M$ depending on $C$, the local boundedness constants and $D$. From these the claim of the Lemma follows by taking an infimum over all $\rho \wedge_{\cU} \cP_{\Gamma}$. The rest of the proof consists of defining $\tilde{\rho}$, showing \eqref{eq:tilderhomassbnd} and proving $\tilde{\rho} \wedge_{\cU} \cP_{\cU,B,L}$.

We do a small preliminary estimate. Fix $V\in \cU$ and any constant $S\geq 1$ and let 
\[
\cU_{V,S} =\{U\in \cU : U \cap B(z_V, (1+2S)r_{V}) \neq \emptyset\}.
\]
 If $U_1,U_2\in \cU_{V,S}$ are distinct, we have by local boundedness 
 \begin{equation}\label{eq:radiusbound}
 \kappa_{1+2S}^{-1} r_{V}\leq r_{U_1},r_{U_2}\leq \kappa_{1+2S}r_{V}.
 \end{equation} 
Thus, the balls $B(z_{U_i},r_{V}\kappa_{1+2S}^{-1}\kappa^{-1})$ are disjoint for $i=1,2$ and are contained in $B(z_V, (1+2S+2\kappa_{1+2S}) r_V)$. Thus, by metric doubling, there are at most $D^m$ sets contained in $\cU_{V,S}$ for any $V\in \cU$ as long as $2^m\geq 4\kappa\kappa_{1+2S}(1+2S+2\kappa_{1+2S}).$

We will next consider $\cU_{V,C}$. Let $L=1+\kappa_{1+2C}(1+2C)$. Choose $k,l\in \N$ with $l\leq k$ and  so that $ 4\kappa\kappa_{1+2C}(1+2C+2\kappa_{1+2C})\leq 2^l$ and so that $ 4\kappa\kappa_{1+2L}(1+2L+2\kappa_{1+2L})\leq 2^k$, and let $M=D^{2kp}$. By the argument after \eqref{eq:radiusbound} with $S=C$, we have that $|\cU_{V,C}|\leq 2^l \leq 2^k$.
Let 
\[
\tilde{\rho}(V)=D^{k} \max\{\rho(U) : U\in \cU_{V,C}\}.
\]
For each $V\in \cU$ choose a $U_V \in \cU_{V,C}$ so that $\tilde{\rho}(V)=D^{k}\rho(U_V)^p$. Let $\tilde{\cU}_{U,C}=\{V\in \cU : U_V=U\}$. For every $V\in \tilde{\cU}_{U,C}$, we have $U_V=U$ and thus $U\in \cU_{V,C}$. Thus $B(z_U,r_U)\cap B(z_V,(1+2C)r_V)\neq  \emptyset$, and $V\subset B(z_V,r_V)\subset B(z_U,r_U+(1+2C)r_V)$. Consequently, from using \eqref{eq:radiusbound} we get $V\cap B(z_U,(1+\kappa_{1+2C}(1+2C))r_U)\neq \emptyset$. In particular, we have $\tilde{\cU}_{U,C}\subset \cU_{U,L}$. Thus, we have by the argument after \eqref{eq:radiusbound} with $S$ replaced with $L$ that $|\tilde{\cU}_{U,C}|\leq D^k$.

Let $P=\{U_1,\dots, U_n\}\in \cP_{\cU,B,L}$. Define a sequence of points $(x_i)_{i=1}^{n+1}$ as follows. Let $x_1\in U_1\cap B,x_{n+1}\in U_{n}\cap X\setminus LB$, and let $x_i\in U_{i}\cap U_{i-1}$ for $i=2,\dots n$. Since $X$ is $C$-LLC, we can find curves $\gamma_i$ connecting $x_i$ to $x_{i+1}$ with $\diam(\gamma_i)\leq Cd(x_i,x_{i+1}) \leq \diam(U_i)$ for $i=1,\dots, n$. Let $\gamma$ be the concatenation of $\gamma_i$. We have that $\gamma \in \Gamma_{B,L}$. 

Now, let $P_{\gamma}=\{U \in \cU: U\cap \gamma \neq \emptyset\}$. We have
\[
\sum_{U\in P_\gamma} \rho(U)\geq 1,
\]
since $\rho \wedge_{\cU} \cP_{\Gamma}$. Now, for each $U\in P_\gamma$ we have some $i=1,\dots, n$ so that $U\cap \gamma_i \neq \emptyset$. Therefore, we have $d(U,U_i)\leq C\diam(U_i)\leq 2Cr_{U_i}$. Thus, $U\in \cU_{U_i,C}$ for some $i$. Let $P_{\gamma,i}=\cU_{U_i,C}\cap P_\gamma$. Since $|\cU_{U_i,C}|\leq D^k$, we get
\[
\tilde{\rho}(U_i) \geq \sum_{U\in P_{\gamma,i}} \rho(U).
\]
Summing these, we get
\begin{align*}
1\leq \sum_{U\in P_\gamma} \rho(U) &\leq \sum_{i=1}^n \sum_{U\in P_{\gamma,i}} \rho(U)\\
&\leq \sum_{i=1}^n \tilde{\rho}(U_i).
\end{align*}
Thus, $\tilde{\rho} \wedge \cP_{\cU,B,L}$, since $P$ was arbitrary.

Finally, we show \eqref{eq:tilderhomassbnd} for $M=D^{k(1+p)}$.  We have by the size bound for $\tilde{U}_{U,C}$ that

\begin{align*}
\sum_{V\in \cU} \tilde{\rho}(V)^p &\leq \sum_{U\in \cU} \sum_{V\in \tilde{\cU}_{U,C}} \tilde{\rho}(V)^p \\
&\leq \sum_{U\in \cU} D^k D^{kp}\rho(U)^p \leq D^{k(1+p)} \sum_{U\in \cU} \rho(U)^p.
\end{align*}

\end{proof}



\subsection{Relationship to Conformal dimension}\label{sec:relationconf}

The proof of the main Theorem \ref{thm:mainthm} is based on the following Theorem of Carrasco-Piaggio. An interested reader may see also \cite{shanmunconf} and \cite{murugan2022conformal} for slightly different versions and proofs of this statement. We have used Lemma \ref{lem:moduli} and Remark \ref{rmk:defapprox} to reformulate the theorem using our notion of $\kappa$-approximations and moduli.

\begin{theorem}[Theorem~1.3 in \cite{carrasco}]\label{thm:confhausdim}Suppose that $X$ is a compact, metrically doubling LLC space, and let $\cU_k$ be $\kappa$-approximations at level $2^{-k}$. Then

\[
\dims_{CAR}(X)=\inf\left\{Q>0 : \liminf_{m\to \infty}\sup_{z\in X, k\geq 0} \Mod_{Q,\cU_{m+k}}(\Gamma_{B(z,2^{-k}),2})=0\right\}.
\]
\end{theorem}

\subsection{New discrete modulus}\label{subsec:newmod}

In this subsection, we introduce a new notion of modulus, which allows for an arbitrary $\kappa$-round collection $\cU$, which may or may not be a $\kappa$-approximation. Indeed, formally we shall permit that the collection even fails to be a cover. This brings the definition closer to that considered by Pansu and Tyson in \cite{pansu, tysonconfdim}. We note that it may be interesting to study more carefully the relationships between their modulus and the one presented here. However, since it would be a side track in the present paper, we do not pursue this here.


\begin{definition}\label{def:stronglyadmissible}
Fix $\tau \geq 4$. Let $\cU$ be a $\kappa$-round, and let $\Gamma$ a family of sets in $X$. We say that $\rho:\cU\to [0,\infty)$ is strongly discretely $\tau$-admissible for $\Gamma$, and write $\rho \stwedge_{\tau, \cU} \Gamma$, if for every $\gamma \in \Gamma$ there exists a collection $\cU_\gamma\subset \cU$ with the following properties:
\begin{enumerate}
\item[i)] $\{B(z_U, \tau r_U): U \in \cU_\gamma\}$ pairwise disjoint; 
\item[ii)] $U\cap \gamma \neq \emptyset$ for all $U\in \cU_\gamma$; and
\item[iii)] we have
\[
\sum_{U\in \cU_\gamma} \rho(U) \geq 1.
\]
\end{enumerate}
\end{definition}

Define
\[
\overline{\Mod}_{p,\tau}(\Gamma, \cU)=\inf_{\rho \stwedge_{\tau,\cU} \Gamma} \sum_{U\in \cU} \rho(U)^p.
\]


The first result we prove, is that the new modulus bounds from above the notion of discrete modulus defined before, when the collections are roughly at the same level.

\begin{proposition}\label{prop:discretemodulus} Let $k\in \N$.  Assume that $X$  is metrically doubling, and that $\kappa \geq 1$, $\tau\geq 4$. There exists a constant $C>0$ so that the following holds for $r>0$.

Suppose that $\cU$ is a $\kappa$-approximation at level $r$ and $\cV$ is a $\kappa$-round collection with $\kappa^{-1}r \leq r_V \leq r$ for every $V\in \cV$. If $\Gamma$ is a collection of curves in $X$, then
\[
\Mod_{p,\cU}(\Gamma) \leq C \overline{\Mod}_{p,\tau}(\Gamma, \cV).
\]
\end{proposition}
\begin{proof} 
Since $X$ is metrically doubling, and by an argument similar to that in Lemma \ref{lem:moduli}, there is a constant $D$ so that each $U\in \cU$ intersects at most $D$ pairwise disjoint sets in $\cV$. Similarly, for each $V\in \cV$ there are at most $D$ many $U\in \cU$ with $U\cap V \neq \emptyset$.

Without loss of generality, assume that $\overline{\Mod}_{p,\kappa,\tau}(\Gamma, \cV)<\infty$. Let $\overline{\rho}\wedge_{\tau, \cV} \Gamma$ be any admissible function so that 
\[
\sum_{V\in \cV} \overline{\rho}(V)^p<\infty.
\]
For each $U\in \cU$, define
\[
\rho(U)=D\max\{\overline{\rho}(V) : U\cap V \neq \emptyset, V\in \cV\},
\]
if there exists some $V\in \cV$ with $U\cap V \neq \emptyset$. If there does not exist any $V\in \cV$ with $V\cap U \neq \emptyset $, set $\rho(U)=0.$

For each $U\in \cU$, for which it is possible, choose one $V_U\in \cV$ so that $U\cap V \neq \emptyset$ and $\rho(U)=D\overline{\rho}(V)$. Let $\cU_V=\{U\in \cU : V_U=U\}$ for $V\in \cV$. We have, by the first paragraph of the proof, that $|\cU_V|\leq D$.

We claim that $\rho \wedge_{\cU} \cP_\Gamma$. Let $\gamma\in \Gamma$ be arbitrary. Let $\cV_\gamma\subset \cV$ be the collection, where each set intersects $\gamma$ and such that the collection $\{B(z_V,\tau r_V): V\in \cV_\gamma\}$ is disjoint with
\[
\sum_{V\in \cV_\gamma}\overline{\rho}(V)\geq 1. 
\]
Let $\cU_\gamma := \{U\in \cU: U\cap \gamma \neq \emptyset\}$. Since $\cU$ is a cover of $X$, for each $V\in \cV_\gamma$, we may choose a $U^V\in \cU$ so that $U \cap (V \cap \gamma) \neq \emptyset$. For each $U\in \cU$, let $\cV_{\gamma, U}=\{V \in \cV_\gamma : U^V=U\}$.   This means that 
\begin{equation}\label{eq:cvgamma}
\cV_\gamma \subset \bigcup_{U\in \cU_\gamma} \cV_{\gamma,U}
\end{equation}
We also have by the first paragraph of the proof that $|\cV_{\gamma, U}|\leq D$ for every $U\in \cU_\gamma$.  Thus, for every $U\in \cU_\gamma$, we have
\begin{equation}\label{eq:rhobarineq}
\rho(U)\geq \sum_{V\in \cV_{\gamma, U}} \overline{\rho}(V).
\end{equation}
By applying \eqref{eq:cvgamma} and \eqref{eq:rhobarineq} we get:
\begin{align*}
\sum_{U\cap \gamma \neq \emptyset} \rho(U) & \geq \sum_{U\cap \gamma \neq \emptyset} \sum_{V\in \cV_{\gamma, U}} \overline{\rho}(V) \\
&\geq \sum_{V\in \cV_\gamma} \overline{\rho}(V) \geq 1.
\end{align*}

Note that $\cU = \bigcup_{V\in \cV} \cU_V$. By using this, we estimate the $p$-energy of $\rho$ using the bound $|\cU_V|\leq D$ for every $V\in \cU$. 
\begin{align*}
\sum_{U\in \cU} \rho(U)^p &\leq \sum_{V \in \cV} \sum_{U \in \cU_V} \rho(U)^p \\
&\leq \sum_{V \in \cV}  D^p |\cU_V|\overline{\rho}(V)^p \\
&\leq \sum_{B \in \cB}  D^{p+1}\overline{\rho}(V)^p.
\end{align*}
Thus, the claim holds for $C=D^{p+1}$ after we take an infimum over $\overline{\rho} \wedge_{\tau, \cV} \Gamma$.
\end{proof}

One  of the benefits of this notion of modulus, is that we can give simple bounds for it in terms of the Hausdorff measure of the space. The following will be an example of such a bound, which we will use. Recall the definition of Hausdorff content $\cH_\delta^p$ from \eqref{def:Hausdorffcontent}.

\begin{proposition}\label{prop:modulusboundhausdorff}Let $\kappa\geq 1, \tau\geq 4$. Let $X$ be any connected compact metric space, and suppose that $\Gamma$ is a family of curves, where each curve in $\Gamma$ is contained in a ball $B(x,R)\subset X$ and has diameter at least $r$. Then, for every $\epsilon\in(0,1), \delta\in (0,\diam(X)/2)$ there exists a $\kappa$-round collection $\cV$ for which
\[
\overline{\Mod}_{p,\tau}(\Gamma, \cV)\leq \frac{(20\tau)^p\cH_\delta^p(B(x,R))+\epsilon}{r^p},
\]
and $\sup_{V\in \cV}r_V\leq \delta$ and each ball in $\cV$ intersects $B(x,R)$ as well as some curve in $\Gamma$.
\end{proposition}
\begin{proof}
Fix $\epsilon>0$. From the definition of Hausdorff content in \eqref{def:Hausdorffcontent}, and by replacing each set in the cover by an enclosing ball, we can find a covering $\cV$ of $B(x,R)$ by balls $V=B=B(z_V, r_V)$ with $r_V\leq \delta$ so that
\[
\sum_{V\in \cV} \diam(V)^p \leq 2^p \cH_\delta^p(B(x,R))+(10\tau)^{-p}\epsilon.
\]
Moreover, by possibly making the collection smaller, we assume that each ball $V\in \cV$ intersects $B(x,R)$ and some curve in $\Gamma$. This modified collection still covers every curve $\gamma\in \Gamma$. Now, $\cV$ is $\kappa$-round with $\kappa=1$. Let $\rho(V)=10\tau\diam(V)/r$. We have by the choice of $\cV$ and $\rho$ that
\[
\sum_{V\in \cV} \rho(V)^p \leq \sum_{V\in \cV} (10\tau)^p \diam(V)^p r^{-p} \leq \frac{(20\tau)^p\cH_\delta^p(B(x,R))+ \epsilon}{r^{p}}.
\]
Therefore, the claim will follow once  we show that $\rho \overline{\wedge}_{\tau,\cV} \Gamma$. Let $\gamma \in \Gamma$. We need to find a collection $\cV_\gamma$ so that the properties i,ii and iii from Definition \ref{def:stronglyadmissible} hold. Since $\cV$ is a cover of $\gamma$, we have that $\{B(z_V,\tau r_V)\}$ is a cover of $\gamma$. Applying the $5r$-covering lemma, we get a finite subcollection $\cV_\gamma \subset \cV$ so that i) $\{B(z_V,\tau r_V) : V\in \cV_\gamma\}$ is pairwise disjoint, ii) so that $\gamma \cap V\neq \emptyset$ for all $V\in \cV_\gamma$ and so that we have that $\cV'=\{B(z_V,5\tau r_V) : V\in \cV_\gamma\}$ is a covering of $\gamma$. Note that 
\[
\diam(B(z_V,5\tau r_V))\leq 10\tau r_V\leq 10\tau \diam(V)=\rho(V)r.
\] 
Since the balls $\{B(z_V,5\tau r_V) : V\in \cV_\gamma\}$ cover $\gamma$, we get
\[
\sum_{V\in \cV, V\cap \gamma \neq \emptyset} \rho(V)\geq \sum_{V\in \cV'} \diam(B(z_V,5\tau r_V))/r \geq \diam(\gamma)/r \geq 1.
\]
Thus, $\rho \overline{\wedge}_{\tau,\cV} \Gamma$ and the claim follows.
\end{proof}

The new notion of modulus is also invariant under quasisymmetries, except for adjusting the $\tau$ parameter. In the following, if $\Gamma$ is a collection of curves in $X$ and $f:X\to Y$ is a homeomorphism, we write $f(\Gamma)=\{f\circ \gamma: \gamma \in \Gamma\}$. The opposite inequality can be obtained by adjusting $\tau$, and applying this lemma to the inverse mapping $f^{-1}$.

\begin{lemma}\label{lem:quasiinv} Let $\tau \geq 4$ and let $f:X\to Y$ be an $\eta$-quasisymmetry.  If $\cV$ is a $\kappa$-round collection, and if $\Gamma$ is any collection of curves in $X$, then
\[
\overline{\Mod}_{p,\tau}(\Gamma, \cV)\leq \overline{\Mod}_{p,\max\{4,\eta(\tau)\}}(f(\Gamma), f(\cV)).
\]
\end{lemma}
\begin{proof} Let $\tau'=\max\{4,\eta(2\tau)\}$.
By Lemma \ref{lem:kapparoundinvariance}, we have that $f(\cV)$ is a $\kappa'$-round collection for some $\kappa'$. Let $\rho \overline{\wedge}_{\tau',f(\cV)} f(\Gamma)$. Define $\overline{\rho}(V)=\rho(f(V))$ for $V\in \cV$. We clearly have
\[
\sum_{V\in\cV} \overline{\rho}(V)^p=\sum_{V\in f(\cV)} \rho(V)^p.
\]
Thus, the claim will follow, if we can show that $\overline{\rho}\overline{\wedge}_{\tau,\cV} \Gamma$.

In the following, elements of $f(\cV)$ will be written as $f(V)$, where $V\in \cV$. Let $\gamma\in \Gamma$. Then, $f\circ \gamma \in f(\Gamma)$ and, since $\rho \overline{\wedge}_{\tau',f(\cV)} f(\Gamma)$, there exists a collection $\cU_{f(\gamma)}\subset f(\cV)$ so that 
\begin{enumerate}
\item[i)] $\{ B(z_{f(V)}, \tau' r_{f(V)}) : f(V)\in \cU_{f(\gamma)}\}$ is pairwise disjoint;
\item[ii)] $f(V)$ intersects $\gamma$ for every $f(V)\in \cU_{f(\gamma)}$;
\item[iii)] we have
\[
\sum_{U\in \cU_{f(\gamma)}} \rho(U)\geq 1.
\]
\end{enumerate}

Let $\cU_\gamma = \{V\in \cV: f(V)\in \cU_{f(\gamma)}\}$. We need to check the three properties from Definition \ref{def:stronglyadmissible} of $\overline{\rho}\overline{\wedge}_{\tau,\cV} \Gamma$:
\begin{enumerate}
\item[a)] $\{ B(z_{V}, \tau r_{V}) : V\in \cU_{\gamma}\}$ is pairwise disjoint;
\item[b)] $V$ intersects $\gamma$ for every $V\in \cU_{\gamma}$;
\item[c)] we have
\[
\sum_{U\in \cU_{\gamma}} \overline{\rho}(U)\geq 1.
\]
\end{enumerate}

From these, b) and c) follow immediately from the properties ii) and iii) above. By Lemma \ref{lem:ballinclusion}, we have
\[
f(B(z_{V}, \tau r_{V}))\subset f(B(z_{f(V)}, \eta(2\tau) r_{f(V)}),
\]
for every $V\in \cU_{\gamma}$. Since $\tau'\geq \eta(2\tau)$, the disjointness in a) follows from that in i).
\end{proof}

\section{Combinatorially Loewner Spaces}\label{sec:combloew}

\subsection{Definition and basic property}
For two closed sets $E,F$, let $\Gamma(E,F)$ be the collection of curves which join them. We adapt the definition of Bourdon and Kleiner of the combinatorial Loewner property slightly, as modified by Clais in \cite[Definition 2.6]{clais}. Let $\cU_k$ be a sequence of $\kappa$-approximations at level $2^{-k}$.

\begin{definition}\label{def:combloew} Fix $p> 1$. We say that a compact LLC space $X$ satisfies the combinatorial $p$-Loewner property, if there exist some increasing continuous functions $\phi,\psi:(0,\infty)\to (0,\infty)$ with $\lim_{t\to 0} \psi(t)=0$, with the following two properties.
\begin{enumerate}
\item For every pair of disjoint continua $E,F\subset X$ and all $k\geq 0$ with $2^{-k}\leq \min\{\diam(E),\diam(F)\}$, we have
\[
\phi(\Delta(E,F)^{-1})\leq \Mod_{p,\cU_k}(\Gamma(E,F)).
\]
\item For every $z\in X$ and $0<r<R$ and all $k\geq 0$ with $2^{-k}\leq r$, we have
\[
\Mod_{p,\cU_k}(\Gamma(\overline{B(z,r)},X\setminus B(z,R)))\leq \psi\left(\frac{r}{R-r}\right).
\]
\end{enumerate}
\end{definition}

Spaces with the combinatorial $p$-Loewner property are also called CLP -spaces or $p$-CLP spaces, if we wish to explicate the exponent $p>1$.

We first note that a combinatorially $p$-Loewner space has conformal Assoad dimension, as well as Ahlfors regular conformal dimension, equal to  $p$. This Lemma is quite well known and is a rather direct consequence from the known Theorem \ref{thm:confhausdim}. However, we present a proof for the sake of clarity, and since its proof does not appear to have been published elsewhere. The proof is very similar, or rather a localized version, of the proof of \cite[Corollary 3.7]{bourdonkleiner}. Later, we will prove Theorem \ref{thm:mainthmcombloew}, which is one of our main contributions, and which improves the following statement by showing that also $\dims_{CH}(X)=p$. 

\begin{lemma}\label{lem:confdimLoew} For a compact LLC space $X$, which is combinatorially $p$-Loewner, it holds that
\[
\dims_{CA}(X)=\dims_{CAR}(X)=p.
\]
\end{lemma}
\begin{proof} Let $\psi$ and $\psi$ be the functions appearing in Definition \ref{def:combloew}.

Since $X$ is a compact LLC space, it is uniformly perfect, and by \cite[Proposition 2.2.6]{MTbook} and \cite[Chapters 14 and 15]{hei01} we have $\dims_{CA}(X)=\dims_{CAR}(X)$. Next, let $\cU_k$ be a sequence of $\kappa$-approximations at levels $2^{-k}$ for $k\in \N$. Let $z\in X$ and $0<r\leq \diam(X)/4$. Then, by the LLC property, there exists a continuum $E\subset \overline{B(z,r)}$ with $\diam(E)\geq r$ and another continuum $F\subset \overline{B(z,3r)}\setminus B(z,2r)$ with $\diam(F)\geq r$. Since every curve connecting $E$ to $F$ contains a sub-curve within $\Gamma(\overline{B(z,r)}, X\setminus B(z,2r))$, we have
\[
\Mod_p(\Gamma(E,F),\cU_k)\leq \Mod_p(\Gamma(\overline{B(z,r)}, X\setminus B(z,2r)),\cU_k). 
\]

However, by the CLP property and since $\Delta(E,F)\leq 6$, we get for all $k\geq 0$ such that $2^{-k}\leq r$ that
\[
\phi(6^{-1})\leq \Mod_p(\Gamma(E,F),\cU_k)\leq  \Mod_p(\Gamma(\overline{B(z,r)}, X\setminus B(z,2r)),\cU_k). 
\]
We thus get:
\[\liminf_{m\to \infty}\sup_{z\in X, k\geq 0} \Mod_{p,\cU_{m+k}}(\Gamma_{B(z,2^{-k}),2})\geq \phi(6^{-1}).
\]
Thus, $p\leq \dims_{CAR}(X)$ by Theorem \ref{thm:confhausdim}.

The inequality $p\geq \dims_{CAR}(X)$ follows by showing that for all $\epsilon>0$, we have 
\[\lim_{m\to \infty}\sup_{z\in X, k\geq 0} \Mod_{p+\epsilon,\cU_{m+k}}(\Gamma_{B(z,2^{-k}),2})=0.
\]
The idea in showing this is to compare the discrete moduli with exponents $p+\epsilon$ and $p$. Indeed, we will show that for all $m\geq 3$ we have
\begin{equation}\label{eq:crucialestimate}
\Mod_{p+\epsilon,\cU_{m+k}}(\Gamma_{B(z,2^{-k}),2})\leq \psi(2^{2-m})^{\epsilon} \Mod_{p,\cU_{m+k}}(\Gamma_{B(z,2^{-k}),2}) \leq \psi(2^{2-m})^{\epsilon} \psi(1). 
\end{equation}
Then, since $\lim_{t\to 0}\psi(t)=0$, the claim follows. 

Let $\rho$ be the optimal function for $\Mod_{p,\cU_{m+k}}(\Gamma_{B(z,2^{-k}),2})$, which exists by Lemma \ref{lem:moduli}. We will show that $\rho(U)\leq \psi(2^{1-m})$ for every $U\in \cU_{m+k}$. This uses a bound for modulus coming from \cite[Lemma 2.3]{bourdonkleiner}, which in turn relies on estimating the modulus of the curves which pass through the set $U$. Let $U\in \cU_{m+k}$. Let $\Gamma_U$ be the collection of curves in $\Gamma_{B(z,2^{-k}),2}$ which intersect $U$. Then any curve in $\Gamma_{B(z,2^{-k}),2}$ which intersects $U$ will contain a sub-curve connecting  $\overline{B(z_U, r_U)}$ to $X\setminus B(z_U, 2^{m-1} r_U)$. Thus, 
\[
\Mod_{p,\cU_k}(\Gamma_U)\leq \Mod_{p,\cU_k}(\Gamma(\overline{B(z_U, r_U)}, X\setminus B(z_U, 2^{m-1} r_U))\leq \psi(2^{2-m}).
\]

By \cite[Lemma 2.3]{bourdonkleiner}, we get for all $U\in \cU_{m+k}$
\[
\rho(U)\leq \Mod_{p,\cU_k}(\Gamma_U)\leq \psi(2^{2-m}).
\]
This, together with the optimality of $\rho$ yields
\[
\sum_{U\in \cU} \rho(U)^{p+\epsilon}\leq \max_{U\in \cU} \rho(U)^\epsilon \sum_{U\in \cU} \rho(U)^{p} \leq  \psi(2^{2-m})^\epsilon \Mod_{p,\cU_{m+k}}(\Gamma_{B(z,2^{-k}),2}),
\]
which is the desired estimate \eqref{eq:crucialestimate}.

\end{proof}

\subsection{Estimates for Modulus}

If the space is combinatorially Loewner, then we can give a lower bound of our modulus, which we introduced in Subsection \ref{subsec:newmod}, in terms of the Bourdon-Kleiner modulus. This is a strengthening of the Proposition \ref{prop:discretemodulus}. In a sense, the following Proposition is the starting point of our paper, since its argument was the first to be discovered.

\begin{proposition}\label{prop:discretemoduluscombloew} Let $k\in \N$, $p>1$.  Assume that $X$  is metrically doubling, LLC and combinatorially $p$-Loewner, and that $\kappa \geq 1$, $\tau\geq 4$. There exists a constant $C>0$ so that the following holds for $r>0$.

Suppose that $\cU$ is a $\kappa$-approximation at level $r$ and $\cV$ is a $\kappa$-round collection with $\inf\{ r_V: V\in \cV\} \geq 2r$. If $\Gamma$ is a collection of curves in $X$ with $2\tau \sup_{V\in \cV} r_V \leq \diam(\gamma)$ for all $\gamma\in \Gamma$, then
\[
\Mod_{p,\cU}(\Gamma) \leq C \overline{\Mod}_{p,\tau}(\Gamma, \cV).
\]
\end{proposition}
\begin{proof}
Again, assume that $\overline{\Mod}_{p,\tau}(\Gamma, \cV)<\infty$, and that $\overline{\rho}\wedge_{\tau, \cV} \Gamma$ with
\[
\sum_{V\in \cV} \overline{\rho}^p(V)<\infty.
\]

For each $V\in \cV$ consider the collection of curves $\Gamma_{V}=\Gamma(\overline{B(z_V, r_V)}, X\setminus B(z_V, (\tau-1) r_V))$. By the $p$-combinatorial Loewner assumption and since $r\leq r_V/2$, we have
\begin{equation}
\Mod_{p,\cU}(\Gamma_V)\leq C,
\end{equation}
for $C=\psi(\frac{1}{\tau-1})>0$, where $\psi$ is from Definition \ref{def:combloew}. Let $\rho_V:\cU\to [0,\infty)$ be such that $\rho_V \wedge_{\cU} \Gamma_V$ and so that
\begin{equation}
\sum_{U\in \cU} \rho_V(U)^p \leq 2C.
\end{equation}

Let 
\[
\rho(U)=\max\{\rho_V(U)\overline{\rho}(V): V\in \cV\}.
\]
We claim that $\rho \wedge_{\cU} \Gamma$. Let $\gamma\in \Gamma$. Since  $\overline{\rho}\wedge_{\tau, \cV} \Gamma$, there exists a collection $\cV_\gamma$ of $V\in \cV$ with 
\begin{enumerate}
\item $V\cap \gamma \neq \emptyset$ for all $V\in \cV_\gamma$;
\item $\{B(z_V,\tau r_V): V\in \cV_\gamma\}$ is a pairwise disjoint collection of balls; and
\item 
\begin{equation}\label{eq:boundrho}
\sum_{V\in \cV_\gamma} \overline{\rho}(V)\geq 1.
\end{equation}
\end{enumerate}

For each $V\in \cV_\gamma$, let $\gamma|_V$ be a minimal subcurve which connected $B(z_V,r_V)$ to $B(z_V,(\tau-1) r_V)$. Such a subcurve exists since $\diam(\gamma)\geq 2\tau r_V$ and $\gamma \cap B(z_V, r_V) \neq \emptyset$. These subcurves are disjoint and $d(\gamma|_V, \gamma|_{V'})\geq 2\min\{r_V,r_{V'}\}\geq 4r$, for distinct $V,V'\in \cV_\gamma$. Therefore, if we let $\cU_{V}=\{U \in \cU: U\cap \gamma|_V\neq \emptyset\}$ for $V\in \cV_\gamma$, then $\cU_{V}\cap \cU_{V'}=\emptyset$ for distinct $V,V'\in \cV_\gamma$. We also have, since $\rho_V \wedge \Gamma_V$ and $\rho \geq \rho_V \overline{\rho}(V)$ that
\begin{equation}\label{eq:lowersumboundgam}
\sum_{U\in \cU_V}\rho(U) \geq \sum_{U\in \cU, U\cap\gamma|_V \neq \emptyset }\rho_V(U)\overline{\rho}(V) \geq \overline{\rho}(V). 
\end{equation}

Now, let $\cU_\gamma =\{U\in\cU: U\cap \gamma \neq \emptyset\}$. We also have
\[
\bigcup_{V\in \cV_\gamma} \cU_V \subset \cU_\gamma.
\]
By the disjointness of the collections $\cU_V$, for distinct $V\in \cV_\gamma$, and by applying \eqref{eq:boundrho}, \eqref{eq:lowersumboundgam} and the choice of $\rho$, we get
\begin{align*}
\sum_{U\in \cU_\gamma}\rho(U) &\geq \sum_{V\in \cV_\gamma} \sum_{U\in \cU_V}\rho(U) \\
&\geq \sum_{V\in \cV_\gamma} \overline{\rho}(V) \geq 1.
\end{align*}
Thus, since $\gamma$ is arbitrary, $\rho \wedge_{\cU} \Gamma$.

Next, we show a mass-bound for $\rho$.  For each $U\in \cU$ let $V_U\in \cV$ be such that $\rho(U)=\rho_{V_U}(U)\overline{\rho}(V_U)$. This yields a partition of $\cU$ into sets $\cU^V=\{U \in \cU : V_U=V\}.$ Thus, we have, since $\cU^V\subset \cU$
 
\begin{align*}
\Mod_p(\Gamma,\cU) \leq \sum_{U\in \cU}\rho(U)^p &= \sum_{V\in \cV} \sum_{U\in \cU^V}\rho_V(U)^p\overline{\rho}(V)^p\\
&\leq \sum_{V\in \cV} \overline{\rho}(V)^p \sum_{U\in \cU}\rho_V(U)^p \\
&\leq \sum_{V\in \cV} 2C \overline{\rho}(V)^p = 2C \sum_{V\in \cV} \overline{\rho}(V)^p.
\end{align*}
By infimizing over $\overline{\rho}$ such that $\overline{\rho}\wedge_{\tau, \cV} \Gamma$ the claim follows.
\end{proof}

We obtain the following proposition, which gives a lower bound for the Hausdorff measure of a combinatorially Loewner space. In this way, this generalizes to combinatorially Loewner spaces the classical estimate of Heinonen and Koskela, \cite[Theorem 3.6]{heinonenkoskela}. That result is much easier to show using continuous modulus. For discrete modulus one needs to do some extra work.

\begin{proposition}\label{prop:Hausdimquant}Let $X$ be a $p$-combinatorially Loewner LLC and metrically doubling space. Then, there exists a constant $C\geq 1$ so that for every $r\in (0,\diam(X))$ and any $x\in X$ we have
\[
\cH^p(B(x,r))\geq C r^p.
\]
\end{proposition}

\begin{proof}
Let $x\in X$. It is sufficient to prove 
\begin{equation}\label{eq:desiredbound}
\cH^p(B(x,2L'r))\geq C r^p.
\end{equation}
for some uniform constants $L'\geq 1,C>0$ for all $ r\in (0,\diam(X)/8)$. Since $X$ is LLC, we can find a continuum $E\subset B(x,r)$ with $r\geq \diam(E)\geq r/2$ and with $x\in E$. Further, there exists a continuum $F \subset \overline{B(x,4r)}\setminus B(x,3r)$ with $8r\geq \diam(F)\geq r$. We have 
\[
1\leq \Delta(E,F)\leq 16.
\]
 Let $\Gamma$ be the collection of continuous curves connecting $E$ to $F$. 

Next, our strategy in proving \eqref{eq:desiredbound} is to show three estimates. We will show that.

\begin{enumerate}
\item[\textbf{A)}] There is a collection $\Gamma_B$ so that for any $\kappa$-approximation $\cU$ at a small enough level the quantity $\Mod_{p,\cU}(\Gamma \setminus \Gamma_B)$ can be bounded from below by using the CLP property, and each curve in $\Gamma \setminus \Gamma_B$ is contained in a ball. 

\item[\textbf{B)}] Proposition \ref{prop:modulusboundhausdorff} gives a lower bound for the Hausdorff measure in terms of the discrete modulus $\overline{\Mod}_{p,\tau}(\Gamma \setminus \Gamma_B,\cV)$. 

\item[\textbf{C)}] Finally, Proposition
 \ref{prop:discretemoduluscombloew} is used  to find a collection $\cU$ for which we can bound $\overline{\Mod}_{p,\tau}(\Gamma_X\setminus \Gamma_B,\cV)$ from below by $\Mod_{p,\cU}(\Gamma_X \setminus \Gamma_B)$ for some $\kappa$-approximation $\cU$ at a small enough level. These estimates together yield the desired bound.
\end{enumerate}

We focus on \textbf{A)} first and determine $\Gamma_B$. Let $\cU$ be a $\kappa$-approximation at level $2^{-k}$ for some $k\in \N$ s.t. $2^{-k}\leq \min\{\diam(E),\diam(F)\}$.
We have
\[
\Mod_{p,\cU}(\Gamma_X) \geq \phi(16^{-1}).
\]
Let $L\geq 2$ be such that $\psi(2L^{-1})\leq 2^{-1} \phi(16^{-1})$. Let $\Gamma_B$ be the collection of curves $\gamma\in \Gamma_X$ with a subcurve in $\Gamma(\overline{B(x,r)}, X\setminus B(x,L r))$. We have, since $X$ is CLP, that 
\[
\Mod_{p,\cU}(\Gamma_B) \leq \Mod_{p,\cU}(\Gamma(\overline{B(x,r)}, X\setminus B(x,Lr))) \leq \psi(2L^{-1})\leq \frac{\phi(16^{-1})}{2}.
\]
Thus, by subadditivity of modulus, we get for $\Gamma_{G}:=\Gamma_X\setminus \Gamma_B$ the estimate
\begin{equation}\label{eq:A)}
\Mod_{p,\cU}(\Gamma_{G}) \geq \frac{\phi(16^{-1})}{2}.
\end{equation}

Next, we deduce \textbf{B)} in our strategy. Let $\tau \geq 4$. Choose $\delta\in(0,4^{-1}\tau^{-1}r)$. Each of the curves in $\Gamma_{G}$ has diameter at least $r$ and is contained in $B(x,Lr)$. So, we can apply Proposition \ref{prop:modulusboundhausdorff} to find for any $\epsilon>0$ a $1$-round collection $\cV$ of balls which intersect $B(x,Lr)$ and some curve in $\Gamma_{G}$ with $\rad(\cV)\leq \delta$ and with 
\begin{equation}\label{eq:B)}
\overline{\Mod}_{p,\tau}(\Gamma_{G}, \cV)\leq (20\tau)^p \left(\cH^p_\delta(B(e,Lr))+\epsilon\right)r^{-p}\leq (20\tau)^p \left(\cH^p(B(e,Lr))+\epsilon\right)r^{-p}.
\end{equation}

Finally, we deduce \textbf{C)}. Each curve in $\Gamma_{G}$ connects $E$ to $F$, and thus we have $\diam(\gamma)\geq r$ for all $\gamma\in\Gamma_G$.
This means that $2\tau \sup_{V\in \cV} r_{V} \leq \inf_{\gamma\in \Gamma_{G}}\diam(\gamma)$.  Thus, by Proposition \ref{prop:discretemoduluscombloew}, there exists a constant $C$ so that
\begin{equation}\label{eq:D)}
\Mod_{p,\cU}(\Gamma_{G}) \leq C \overline{\Mod}_{p,\tau}(\Gamma_{X,G}, \cV),
\end{equation}
if $k$ is so large so that $\inf\{r_{V} : V\in \cV\} \geq 2^{-k-1}$. 

By combining Estimates \textbf{A-C)}, we get the following once $k$ is large enough
\begin{align*}
\frac{\phi(\eta(16)^{-1})}{2} & \stackrel{\eqref{eq:A)}}{\leq} \Mod_{p,\cU}(\Gamma_{G}) \\
&\stackrel{\eqref{eq:D)}}{\leq} C \overline{\Mod}_{p,\tau}(\Gamma_{G}, \cV) \\
&\stackrel{\eqref{eq:B)}}{\leq} (20\tau)^p C (\cH^p(B(x,Lr))+\epsilon)r^{-p}.
\end{align*}
Consequently, since this holds for all $\epsilon>0$, we get
\[
\frac{\phi(\eta(16)^{-1})}{2  (20\tau)^p C}r^p\leq \cH^p(B(x,L'r)).
\]
This yields the desired estimate \eqref{eq:desiredbound}.

\end{proof}

\subsection{Proof of Theorem \ref{thm:mainthmcombloew}}

Using the previous properties, we are able to prove the equality of different forms of conformal dimension for CLP spaces.

\begin{proof}[Proof of Theorem \ref{thm:mainthmcombloew}]
Assume that $X$ is combinatorially $p$-Loewner.
By Lemma \ref{lem:confdimLoew}, we have
\[
\dims_{CA}(X)=\dims_{CAR}(X)=p.
\]
We also have $\dims_{CH}(X)\leq \dims_{CAR}(X)=p$. 
Thus, we only need to show that $\dims_{CH}(X)\geq p$. Let $f:X\to Y$ be a quasisymmetry. The space $Y$ is $p$-combinatorially Loewner, since the combinatorial Loewner property is invariant under quasisymmetries, see \cite[Theorem 2.6 (2)]{bourdonkleiner}. It is also easy to see, that the LLC and metric doubling properties are invariant under quasisymmetries, and thus $Y$ is LLC and metric doubling. Then, by Proposition \ref{prop:Hausdimquant} there exists a constant $C$ so that we have for every $y\in Y$ and any $r\in (0,\diam(Y))$ that
\[
\cH^p(B(y,r))\geq C r^p >0.
\]
From the definition of Hausdorff dimension, and since $\cH^p(B(y,r))>0$ if and only if $\cH^p_\infty(B(y,r))>0$, we have $\dims_H(Y)\geq p$. Consequently, by taking an infimum over all $Y$ which are quasisymmetric to $X$, we get $\dims_{CH}(X)\geq p$.
\end{proof}

\section{Quasiself-similar Spaces}\label{sec:quasiselfsim}

\subsection{Uniform bound for Annuli}

Define an annulus as $A(x,r,R):=B(x,R)\setminus \overline{B(x,r)}$. 

\begin{definition}\label{def:smallannuli}Let $p\in (1,\infty)$. Let $\tau\geq 4$. We say that a metric space $X$ has uniformly small $p$-moduli of annuli, if there exists $\epsilon\in (0,1)$ and constants $0<\delta_{-}<\delta_+ < \tau^{-1}$, so that the following holds. For every annulus $A(x,r,(\tau-2)r)$ in $X$, with $x\in X, r \in (0,2^{-1}\tau^{-1}\diam(X))$, there exists a finite collection of balls $\cV_{x,r}$ contained in $B(x,\tau r)$ and which intersect $B(x,(\tau-2)r)$, with $r_V \in [\delta_- r, \delta_+ r]$ for each $V\in \cV_{x,r}$, and there exists a function $\rho_{x,r}:\cV_{x,r}\to[0,\infty)$ with 
\[
\rho_{x,r} \overline{\wedge}_{\tau,\cV_{x,r}} \Gamma(\overline{B(x,r)},X \setminus B(x,(\tau-2)r))
\] and with
\[
\sum_{B\in \cV_B} \rho_{x,r}(B)^p \leq \epsilon.
\]
\end{definition}

The following Lemma is a refinement of Proposition \ref{prop:modulusboundhausdorff} to the quasiself-similar setting. 

\begin{lemma}\label{lem:Hausdorff} Suppose that $\dims_{CH}(X)<p$, that $p\in (1,\infty)$, and that $X$ is an arcwise connected quasiself-similar compact metric space. Then $X$ has uniformly small $p$-moduli of annuli.
\end{lemma}

\begin{proof} Assume that $X$ is $\eta$-quasiself-similar and let $\tau \geq 4$. Fix any $\delta_+\in (0,\tau^{-1})$. Since $\dims_{CH}(X)<p$, there exists a compact space $Y$ with $\dims_H(Y)<p$ and a quasisymmetry $g:X\to Y$. Fix $C\geq 1,\sigma\in (0,2^{-1})$ to be determined. By adjusting $\eta$, we may assume that $g$ is an $\eta$-quasisymmetry. Let $\epsilon>0$, and choose a covering of $Y$ by a collection of balls $\cB_Y$ with
\[
\sum_{B\in \cB_Y} \diam(B)^p \leq \epsilon C^{-p}\diam(Y)^p,
\]
and for which $\rad(B)\leq \sigma \diam(Y)$ for every $B\in \cB_Y$. Let $A(x,r,(\tau-2)r)$ be an annulus in $X$ with $x\in X$ and $r \in (0,2^{-1}\tau^{-1}\diam(X))$.  There is a homeomorphism $f:B(x,2\tau r)\to U$, for some open set $U\subset X$, which is an $\eta$-quasisymmetry, where $\diam(U)\geq \delta \diam(X)$. 

We first define the collection $\cV_{x,r}$ used in Definition \ref{def:smallannuli}.
For each $B=B(y,s)\in \cB_Y$ with 
\[
B\cap g(f(B(x,(\tau-2)r)))\neq \emptyset,
\] choose $x_{V_B}\in (g\circ f)^{-1}(B)\cap B(x,(\tau-2)r)$, and let $r_{V_B}=\sup\{d(x,x_B): x\in (g\circ f)^{-1}(2B)\cap B(x,\tau r)\}$. Define $V_B:=B(x_{V_B},r_{V_B})$. 
Let 
\[
\cV_{x,r}:=\{V_B : B\in \cB_Y, B \cap g(f(B(x,(\tau-2)r))) \neq \emptyset\}
\]
be the collection of balls we seek. Next, we give bounds for $r_{V_B}$ by using the fact that $X$ is connected and that $g\circ f$ is a $\tilde{\eta}$-quasisymmetry with $\tilde{\eta}=\eta\circ \eta$. 

Since $\diam(U)\geq \delta \diam(X)$, we can choose $a,b\in U$ with $d(a,b)\geq 2^{-1}\delta \diam(X)$. Choose a point $c\in X$ so that $d(g(c),g(a))\geq \diam(Y)2^{-1}$. Since $g$ is an $\eta$-quasisymmetry, we have 
\[
\frac{d(g(a),g(c))}{d(g(a),g(b))}\leq \eta\left(\frac{d(c,a)}{d(b,a)}\right)\leq \eta(2\delta^{-1}).
\]
Thus, 
\begin{equation}\label{eq:gagbbound}
d(g(a),g(b))\geq \eta(2\delta^{-1})^{-1}2^{-1}\diam(Y).
\end{equation}

We will use \eqref{eq:gagbbound} to give an upper bound for $r_{V_B}$ for each $V_B\in \cV_{x,r}$, where $B\in \cB_Y$. Let $u,v\in B(x,2\tau r)$ be such that $f(u)=a,f(v)=b$. Choose $s,t\in  (g\circ f)^{-1}(2B)\cap B(x,\tau r)$ so that $d(s,t)\geq r_{V_B}/2$. Up to possibly switching $u$ and $v$, and $a,b$, we can assume by \eqref{eq:gagbbound} that  
\begin{equation}\label{eq:gsgabound}
d(g(f(s)),g(a))\geq \frac{d(g(a),g(b))}{2}\geq \eta(2\delta^{-1})^{-1}\diam(Y)2^{-2}.
\end{equation}

We have
\begin{equation}\label{eq:eqsust}
\frac{d(g(f(s)),g(f(u)))}{d(g(f(s)),g(f(t)))}\leq \tilde{\eta}\left(\frac{d(s,u)}{d(s,t)}\right).
\end{equation}
Since $g(f(s)),g(f(t))\in 2B$, we get $d(g(f(s)),g(f(t)))\leq 4\rad(B)\leq 4\sigma \diam(Y)$. Thus, from \eqref{eq:gsgabound}, we get
\[
\frac{1}{2^4 \eta(2\delta^{-1})\sigma}\leq \frac{\diam(Y)}{2^4\rad(B)\eta(2\delta^{-1})}\leq \frac{d(g(f(s)),g(a))}{d(g(f(s)),g(f(t)))}.
\]
By combining this with \eqref{eq:eqsust}, we deduce 
\[
\tilde{\eta}^{-1}\left(\frac{1}{2^4\sigma \eta(2\delta^{-1})}\right)\leq \frac{d(s,u)}{d(s,t)} \leq \frac{4\tau r}{r_{V_B}}.
\]
Thus,
\[
r_{V_B} \leq  \frac{4\tau}{\tilde{\eta}^{-1}\left(\frac{1}{2^4\sigma \eta(2\delta^{-1})}\right)} r.
\]
Choose now $\sigma \leq \tilde{\eta}(\frac{4\tau}{\delta_+})^{-1}\eta(2\delta^{-1})^{-1}2^{-4}$. We then have, $r_{V_B}\leq \delta_+ r$. Since $\delta_+<1$, we also have $r_{V_B}\leq r$ and since $x_{V_B}\in B(x,(\tau-2)r)$ we clearly have $V_B\subset B(x,\tau r)$. 

Next, we give a uniform lower bound for the radii $r_{V_B}$ for $V_B\in \cV_{x,r}$ . Since $\cB_Y$ is finite, there exists a constant $\beta>0$ so that $\rad(B)\geq \beta \diam(Y)$ for all $B\in \cB_Y$.  Choose $\delta_-=\tilde{\eta}^{-1}(\beta)/2$. Let $c\in B(x_{V_B}, \delta_- r)$ be an arbitrary point. Also, choose $b\in B(x,2\tau r)$ with $d(b,x_{V_B})\geq r$, which is possible by connectivity. Then, by the quasisymmetry condition, we get
\[
\frac{d(g(f(c)), g(f(x_{V_B})))}{d(g(f(b)), g(f(x_{V_B})))}\leq \tilde{\eta}\left(\frac{d(c,x_{V_B})}{d(b,x_{V_B})}\right) \leq \tilde{\eta}\left(\delta_-\right).
\]
The choice of $\delta_-$ guarantees $\tilde{\eta}(\delta_-)\leq \beta$, and thus
\[
d(g(f(c)), g(f(x_{V_B}))) \leq \tilde{\eta}(\delta_-)\diam(Y)\leq r_B.
\]
Therefore, since $g(f(x_{V_B}))\in B$, we get $g(f(c))\in 2B$. This holds for all $c\in B(x_{V_B}, \delta_- r)$, and thus 
\[
g(f(B(x_{V_B}, \delta_- r))) \subset 2B.
\]
This yields, by connectivity and the definition of $r_{V_B}$ that $r_{V_B}\geq \delta_- r$. 

Finally, we define the admissible function $\rho$. Define $\rho(V)=\max\{C \diam_Y(B)\diam_Y(Y)^{-1}: V_{B}=V, B\in \cB_{Y}\}$ for $V\in \cV_{x,r}$.
We have
\begin{equation} \label{eq:massbound}
\sum_{V\in \cV_B} \rho(V)^p \leq C^p \sum_{B\in \cB_Y} \diam_Y(B)^p\diam_Y(Y)^{-p} \leq \epsilon,
\end{equation}
since for every $V\in \cV_{x,r}$ there exists at least one $B\in \cB_Y$ so that $V_B=V$, and for every $B\in \cB_Y$ there is only one $V\in\cV_{x,r}$ for which $V_B=V$.

Next, we show that $\rho \overline{\wedge}_{\tau, \cV_{x,r}} \Gamma(\overline{B(x,r)},X\setminus B(x,(\tau-2)r))$. Let $\gamma \in \Gamma(\overline{B(x,r)},X\setminus B(x,(\tau-2)r))$ be arbitrary. Let $\sigma$ be a sub-curve of $\gamma$ so that $\sigma\subset \overline{(\tau-2)B}$ and $\sigma\in \Gamma(\overline{B(x,r)},X\setminus B(x,(\tau-2)r))$. To show admissibility, we will combine the fact that $\cB_Y$ covers $g\circ f\circ \sigma$ with a lower bound for the diameter of $g\circ f\circ \sigma$. 

Since $\sigma$ connects $B(x,r)$ to $X\setminus B(x,(\tau-2)r)$ there exist $j,k\in \sigma$ with $d(j,k)\geq (\tau-3)r$.  Let $a,b$ and $u,v$ be as before. By possibly switching $j$ and $k$, we can assume that $d(j,u)\geq d(j,k)/2 \geq 2^{-1}(\tau-3)r$. We get
\[
\frac{d(g(f(j)),g(f(u)))}{d(g(f(j)),g(f(k))} \leq \tilde{\eta}\left(\frac{d(j,u)}{d(j,k)}\right)\leq \tilde{\eta}\left(\frac{4\tau r}{(\tau-3)r}\right)\leq \tilde{\eta}(16). 
\]
Thus,
\begin{equation}\label{eq:diamsigma-1}
\diam(g\circ f\circ \sigma) \geq d\left((g(f(j)),g(f(k))\right) \geq d\left(g(f(j)),g(f(u))\right)\tilde{\eta}(16)^{-1}.
\end{equation}
Next, $d(j,u)\geq 2^{-1}(\tau-3)r \geq d(u,v)/8$. Thus, by a similar reasoning that uses the quasisymmetry of $g\circ f$ and by employing \eqref{eq:gagbbound}, we get
\begin{equation}\label{eq:diamsigma-2}
d(g(f(j)),g(f(u))) \geq d(g(f(u),g(f(v)))\tilde{\eta}(8)^{-1}\geq \tilde{\eta}(48)^{-1}\eta(2\delta^{-1})^{-1}2^{-1}\diam(Y).
\end{equation}
By combining \eqref{eq:diamsigma-1} and \eqref{eq:diamsigma-2}, we obtain
\begin{equation}\label{eq:diamsigma-3}
\diam(g\circ f\circ \sigma) \geq d(g(f(j)),g(f(k)) \geq \tilde{\eta}(16)^{-1} \tilde{\eta}(8)^{-1}\eta(2\delta^{-1})^{-1}2^{-1}\diam(Y).
\end{equation}

Recall that $\cV_{x,r}$ consists of balls.
The open sets $\cV_{x,r}$ cover the ball $B(x,2(\tau-2))$, and thus the curve $\sigma$. Therefore, by the Vitali covering theorem, there exists a finite collection of balls $\cV_\gamma$ with
$\sigma \subset \bigcup 5\tau \cV_\gamma$, and for which $\tau \cV_\gamma$ are disjoint, and so that each ball in $\cV_\gamma$ intersects $\gamma$.  For each $V\in \cV_\gamma$, choose a ball $B(V)\in \cB_Y$ so that $V=V_{B(V)}$ and $\rho(V)=C\diam_Y(B(V))\diam_Y(Y)^{-1}$. 

First, we note that the quasisymmetry condition  and Lemma \ref{lem:ballinclusion}, we have $g(f(5\tau V)) \subset \tilde{\eta}(10\tau) B(V)$. Therefore, we get that the balls $\tilde{\eta}(10\tau) B(V)$ for $V\in \cV_\gamma$ cover $g(f(\sigma))$. Thus,
\begin{align*}
\sum_{V\in \cV_\gamma} \rho(V) &=\sum_{V\in \cV_\gamma} C\diam_Y(B(V))\diam_Y(Y)^{-1} \\
&\geq \sum_{V\in \cV_\gamma} C(2\tilde{\eta}(10\tau))^{-1}\diam_Y(Y)^{-1}\diam_Y(\tilde{\eta}(10\tau) B(V)) \\
&\stackrel{\eqref{eq:diamsigma-3}}{\geq} \diam(g\circ f\circ \sigma) C(2\tilde{\eta}(10\tau))^{-1}\diam_Y(Y)^{-1} \\
&\geq C 2^{-2}\tilde{\eta}(16)^{-1} \tilde{\eta}(8)^{-1}\eta(2\delta^{-1})^{-1}\tilde{\eta}(10\tau)^{-1}.
\end{align*}

If $C\geq 4\tilde{\eta}(16) \tilde{\eta}(8)\eta(2\delta^{-1})\tilde{\eta}(10\tau)$, then $\rho \overline{\wedge}_{\tau, \cV_{x,r}} \Gamma(\overline{B(x,r)},X\setminus B(x,(\tau-2)r))$ is admissible and the claim follows.

 
\end{proof}

\subsection{Algorithm for pushing down a cover}\label{subsec:algorithm}

The following lemma describes a ``push down'' algorithm. It uses admissible functions for annuli in order to push down a collection of balls $\cB$ and a strongly discretely $\tau$-admissible function $\rho$. This is done by replacing a ball $\mathbf{B}\in \cB$ by a collection $\cB_\mathbf{B}$ and an associated function $\rho_\mathbf{B}$. A new admissible function $\overline{\rho}$ is defined by taking a maximum over $\mathbf{B}\in \cB$, and a new collection by taking a union of all the new balls. This arguments for admissibility and the construction of $\overline{\rho}$ are similar to Proposition \ref{prop:discretemoduluscombloew}. 
To distinguish the ``parent'' balls from the ``descendant balls'', we will bold the parent balls. This replacement algorithm is depicted and explained more in Figure \ref{fig:replacement}. As seen in this figure, we permit all sorts of overlaps, and balls of different sizes. This is one of the technical reasons for using the new modulus from Subsection \ref{subsec:newmod}.

Recall that $\Gamma_{L,B}$ denotes the collection of curves $\gamma$ connecting $B$ to $X\setminus LB$.

\begin{lemma}\label{lem:replaceannulus}Let $\epsilon,\eta \in (0,1)$. Assume that $\cB$ is a finite collection of balls, $\Gamma$ is a collection of curves, $2(\tau-2)\rad(\cB)\leq \inf_{\gamma\in  \Gamma}\diam(\gamma)$ and $\rho \wedge_{\cB} \Gamma$. Suppose further that $\cC\subset \cB$ is any finite collection of balls, and for every  $\mathbf{B}\in \cC$, there exists a finite collection of balls $\cB_\mathbf{B}$ and a function $\rho_\mathbf{B}:\cB_\mathbf{B}\to [0,\infty)$ with 
\begin{enumerate}
\item $\rad(\cB_\mathbf{B})\leq \tau^{-1}\rad(\mathbf{B})$,
\item  $\rho_\mathbf{B}\overline{\wedge}_{\tau,\cB_\mathbf{B}} \Gamma_{\mathbf{B}, (\tau-2)}$, 
\item every ball in $\cB_\mathbf{B}$ intersects $(\tau-2)\bfB$, and satisfies
\[
\sum_{B\in \cB_\mathbf{B}} \rho_\mathbf{B}(B)^p \leq \eta.
\] 
\end{enumerate}
For $\mathbf{B}\not\in \cC$ assume that $\cB_\mathbf{B}=\{\mathbf{B}\}$ and $\rho_\mathbf{B}(\mathbf{B})=1$.

For the collection $\overline{\cB} := \bigcup_{\mathbf{B}\in \cB} \cB_\mathbf{B}$, and function
\[
\overline{\rho}(B):=\max\{\rho(\mathbf{B})\rho_\mathbf{B}(B): \bfB\in \cB \ \text{s.t.} \ B\in \cB_\mathbf{B}\},
\] we have $\overline{\rho} \overline{\wedge}_{\tau,\overline{\cB}} \Gamma$ and
\[
\sum_{B\in \overline{\cB}} \overline{\rho}(B)^p \leq \sum_{\mathbf{B}\in \cC} \eta \rho(\mathbf{B})^p + \sum_{\mathbf{B}\in \cB \setminus \cC} \rho(\mathbf{B})^p .
\]
\end{lemma}

\begin{figure}[h!]\label{fig:replacement}\includegraphics[width=.3\textwidth]{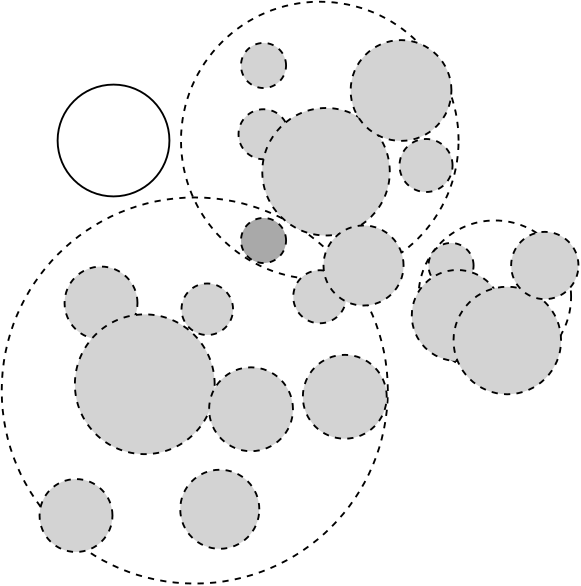}
\caption{The push-down algorithm: As input we are given a collection of balls $\cB$, and a function $\rho$ which is strongly discretely admissible for some collection of curves $\Gamma$. For each ball $B$ in some sub-collection $\cC\subset \cB$ we have a ``replacement collection'' $\cB_\textbf{B}$, and a replacement function $\rho_\textbf{B}$. We replace each ball in $\cC$ by the associated collection $\cB_B$, and keep the remaining balls fixed. A new $\overline{\rho}$ is defined using $\rho_\textbf{B}$ for this new collection, and is shown to be admissible under appropriate assumptions. In the figure, the non-filled balls are the collection $\cB$, and $\cC$ consists of the empty balls with dashed boundaries. The figure shows how the dashed lined empty balls are replaced by the filled balls, where each filled ball comes from a collection $\cB_\textbf{B}$ for some $\textbf{B}\in \cC$. We have not indicated in detail which ball corresponds to which, as this is irrelevant for the proof. Note however that condition (3) in Lemma \ref{lem:replaceannulus}. There is a darker shaded ball, which is in two of the collections $\cB_\textbf{B}$ -- an eventuality that we permit, but that could be avoided by  allowing collections of balls to be multi-sets. Also note, that the balls can intersect each other in unlimited ways, and there is also one solid lined ball in $\cB\setminus \cC$ which is not replaced. }
\end{figure}

\begin{proof}[Proof of Lemma \ref{lem:replaceannulus}] We first show that $\overline{\rho} \stwedge_{\tau,\overline{\cB}} \Gamma$. Let $\gamma\in \Gamma$. Since $\rho \stwedge_{\tau,\cB} \Gamma$, there exists a collection $\cB_\gamma \subset \Gamma$ so that $\tau\cB_\gamma$ is pairwise disjoint, so that $\textbf{B}\cap\gamma\neq \emptyset$ for every $\textbf{B}\in \cB_\gamma$ and
\begin{equation}\label{eq:rhoadm}
\sum_{\textbf{B}\in \cB_\gamma} \rho(\bfB)\geq 1.
\end{equation}

We next define $\overline{\cB}_\gamma\subset \overline{\cB}$. First, set $\cB_\gamma^1 = \cB_\gamma \setminus \cC$. Next, for each $\bfB \in \cB_\gamma \cap \cC$ we have $\rho_\bfB \stwedge_{\tau,\cB_\bfB} \Gamma_{\bfB,(\tau-2)}$. Since $\gamma \cap \bfB\neq \emptyset$, and $\diam(\gamma)\geq 2(\tau-2)\rad(\bfB)$, we have that $\gamma \cap (X\setminus (\tau-2)\bfB)\neq \emptyset$. Thus, $\gamma$ contains a sub-arc in $\Gamma_{\bfB, (\tau-2)}$. Therefore, there exists a collection $\cB_{\gamma,\bfB}\subset \cB_\bfB$ so that $\tau\cB_{\gamma,\bfB}$ is pairwise disjoint, so that $B\cap\gamma\neq \emptyset$ for every $B\in \cB_{\gamma,\bfB}$ and
\[
\sum_{B\in \cB_{\gamma,\bfB}} \rho_\bfB(B)\geq 1.
\]
Since $\overline{\rho}(B)\geq \rho_\bfB(B)\rho(\bfB)$ for every $B\in \cB_{\gamma, \bfB}$, we have
\begin{equation}\label{eq:rhoadmann}
\sum_{B\in \cB_{\gamma,\bfB}} \overline{\rho}(B)\geq \rho(\bfB).
\end{equation}

Set $\cB_\gamma^2 = \bigcup_{\bfB\in \cB_\gamma \cap \cC} \cB_{\gamma,\bfB}$. Finally, let $\overline{\cB}_\gamma=\cB_\gamma^1 \cup \cB_\gamma^2$. Note that for every $\bfB\in \cB_\gamma \cap\cC$, we have $\rad(\tau \cB_{\gamma,\bfB})\leq \tau \rad(\cB_\bfB)\leq \rad(B)$ and $B\cap (\tau-2)\bfB \neq \emptyset$ for every $B\in \cB_{\gamma,\bfB}$. 
Thus, every $B\in \tau \cB_{\gamma,\bfB}$ satisfies $\tau B \subset \tau \bfB$. This inclusion implies that the collections $\cB_{\gamma,\bfB}$ are pairwise disjoint for distinct $\bfB\in \cB_\gamma$, and each of these is disjoint from $\cB_\gamma^2$. Thus, the collection $\tau \overline{\cB}_\gamma$ is pairwise disjoint.

Next, 
\begin{align*}
\sum_{B\in \overline{\cB}_\gamma} \overline{\rho}(B) &= \sum_{B\in \cB_\gamma^1} \overline{\rho}(B) + \sum_{B\in \cB_\gamma^2} \overline{\rho}(B) \\
&\geq \sum_{B\in \cB_\gamma^1} \rho(B) + \sum_{\bfB\in \cB_\gamma \cap\cC}\sum_{B\in \overline{\cB}_{\gamma, \bfB}} \overline{\rho}(B) \\
&\stackrel{\eqref{eq:rhoadmann}}{\geq} \sum_{\bfB\in \cB_\gamma \setminus \cC} \rho(\bfB) + \sum_{\bfB\in \cB_\gamma \cap\cC} \rho(\bfB) \\
&\stackrel{\eqref{eq:rhoadm}}{=} \sum_{\bfB\in \cB_\gamma} \rho(\bfB) \geq 1.
\end{align*}
Thus, since $\gamma$ was arbitrary, we have $\overline{\rho}\overline{\wedge}_{\tau, \overline{\cB}} \Gamma.$

Finally, we compute the $p$-energy of $\overline{\rho}$. 
First, by construction, for every $\bfB\in \cC$, we have
\begin{equation}\label{eq:tilderhosum}
\sum_{B \in \cB_\bfB} \rho_\bfB(B)^p\leq \eta,
\end{equation}
and for every $\bfB \in \cB \setminus \cC$, we have
\begin{equation}\label{eq:tilderhosum2}
\sum_{B \in \cB_\bfB} \rho_\bfB(B)^p=\rho_\bfB(\bfB)=1,
\end{equation}
For every $B\in \overline{\cB}$, there may be multiple $\bfB\in \cB$ so that $B\in \cB_\bfB$. However, for every $B\in \overline{\cB}$, we have:
\begin{equation}\label{eq:upperboundtildeB}
\overline{\rho}(B) = \max\{\rho(\bfB)\rho_\bfB(B) : \bfB\in \cB \ \text{s.t.} \ B \in \cB_\bfB\} \leq \left(\sum_{\substack{\bfB\in \cB \\ \text{s.t.} \ B\in \cB_\bfB}} (\rho(\bfB)\rho_\bfB(B))^p\right)^{\frac{1}{p}}.
\end{equation}
By combining these two we get:
\begin{align*}
\sum_{B\in \overline{\cB}} \overline{\rho}(B)^p &\stackrel{\eqref{eq:upperboundtildeB}}{\leq} \sum_{\bfB\in \cB} \sum_{B\in \cB_\bfB} \rho(\bfB)^p\rho_\bfB(B)^p  \\
&=\sum_{\bfB\in \cC} \sum_{B\in \cB_\bfB} \rho(\bfB)^p\rho_\bfB(B)^p + \sum_{\bfB\in \cB\setminus \cC} \sum_{B\in \cB_\bfB} \rho(\bfB)^p\rho_\bfB(B)^p  \\
&\stackrel{\eqref{eq:tilderhosum},\eqref{eq:tilderhosum2}}{\leq} \sum_{\bfB\in \cC} \eta\rho(\bfB)^p + \sum_{\bfB\in \cB\setminus \cC} \rho(\bfB)^p.
\end{align*}
This is the desired estimate and the proof is complete.
\end{proof}

If the space has uniformly small moduli of annuli, then we have readily available collections $\cB_\bfB$ and functions $\rho_\bfB$ as in the statement of the previous lemma. Suppose now that $\Gamma$ is any collection of curves and $\rho \overline{\wedge}_{\tau, \cB_0} \Gamma$ for some collection of balls $\cB_0$ with sufficiently small radii compared to the diameter of the curves $\gamma\in \Gamma$.  If the space has uniformly small moduli of annuli, we can first use the push-down procedure on all of the balls, in order to reduce the discrete modulus of $\Gamma$ below any threshold we want. This process is called ``weight reduction''. The balls in the resulting collection can have arbitrarily large and small radii. However, any ``large'' balls can be replaced by smaller ones with controlled size. This process can be repeated until all the balls in a collection are roughly the same size. This phase is called ``size reduction''. As a consequence, the push-down algorithm in conjunction with the uniformly small moduli of annuli leads to the ability to convert the collection $\cB$ into a collection of balls $\overline{\cB}$ which is roughly at the same scale, and so that $\overline{\Mod}_{p,\tau}(\Gamma,\overline{\cB})$ is as small as we want. This process is called the ``equalizing algorithm'', and is depicted in Figure \ref{fig:equalizing}. 

\begin{figure}[h!]\label{fig:equalizing}\includegraphics[width=.4\textwidth]{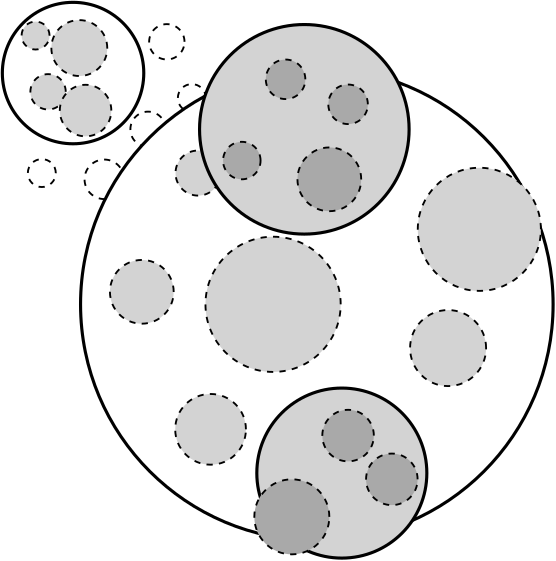}
\caption{The equalizing algorithm: By using replacement and a uniform bound on moduli of annuli, we can ``uniformize'' a wild cover $\cB$. Let $\cB$ be a covering using balls, where the size of the largest ball is much bigger than the smallest. We take all the ``large'' balls, and form a collection $\cC$ of them. To them, we apply the push-down procedure to reduce their size. We repeat this process until all large balls have been pushed down to a size comparable to the smallest ball in our collection. In the figure $\cB$ consists of balls filled with white. The two large balls have solid line boundaries, and are replaced by smaller light gray filled balls. Two of these light gray balls are still too large, and are replaced by even smaller dark gray filled balls.}
\end{figure}

The next lemma and its proof expresses formally the effect of the equalizing algorithm.

\begin{lemma}\label{lem:equalizingalg}Let $p\in (1,\infty), \tau\geq 4$. Let $X$ be an LLC compact metric space which has uniformly small $p$-moduli of annuli, with constant $\delta_->0$ as in Definition \ref{def:smallannuli}. Then there exists a $\kappa\geq 1$ so that for every $\epsilon_0>0$ and every $M\geq 1$ there exists a $N\geq 1$ with the following property. Let $\cB_0$ be a collection of balls for which $\inf_{\gamma\in \Gamma}\diam(\gamma)\geq \tau \rad(\cB_0)$,  and
\[
\overline{\Mod}_{p,\tau}(\Gamma, \cB_0)< M.
\]

Then, for every  $r>0$ with $\delta_-^{N}\inf\{\rad(B): B\in \cB_0\}\geq r$, there exists a collection of balls $\cV$ with
\[
\overline{\Mod}_{p,\tau}(\Gamma,\cV)\leq \epsilon.
\]
and  $\kappa^{-1} r\leq r_V\leq r$ for all $V\in \cV$.

\end{lemma}

We briefly remark, that is is crucial for our later arguments, that $\kappa$ is independent of $\epsilon_0>0$. That is, there exists a $\kappa$, which controls the ratio of the largest ball in $\cV$ to the smallest one. Given this $\kappa$, we can choose any $\epsilon$ to obtain a modulus smaller than that. The cost of decreasing $\epsilon$ is to increase $N$ and thus reduce the scale $r$ of $\cV$.

\begin{proof}[Proof of Lemma \ref{lem:equalizingalg}]
Let $\tau \geq 4$, and let $0<\delta_-<\delta_+<\tau^{-1}$ and $\epsilon\in (0,1)$ be the parameters from the definition of having uniformly small $p$-moduli of annuli, see Definition \ref{def:smallannuli}. Let $\kappa =2\delta_-^{-1}$,  $N=\lceil \log_{2}(\max\{M/\epsilon_0,1\})/\log_2(\epsilon^{-1}) \rceil + 1$ and $r\leq \delta_-^{N}\inf\{\rad(B): B\in \cB_0\}$.

Further, for any ball $\bfB=B(x,s)$ in $X$, let $\cB_{\bfB}$ and $\rho_B$ denote the functions $\cV_{x,s}$ and $\rho_{x,s}$ given in Definition \ref{def:smallannuli} with
\begin{equation}\label{eq:smallanncond}
\sum_{B\in \cB_\bfB} \rho_{\bfB}(B)^p \leq \epsilon,
\end{equation}
and $\rho \overline{\wedge}_{\tau, \cB_\bfB} \Gamma_{\bfB, \tau-2}.$

Let $\cB_0$ be a finite cover of $X$ by balls with $\inf_{\gamma\in \Gamma}\diam(\gamma)\geq \tau \rad(\cB_0)$ for which $\overline{\Mod}_{p,\tau}(\Gamma, \cB_0)< M$. Then, there exists a $\rho_0:\cB_0\to [0,\infty)$ with $\rho \overline{\wedge}_{\tau,\cB_0} \Gamma$ and with mass
\[
\sum_{B\in \cB_0} \rho(B)^p <M.
\] 

First, we replace $\cB_0$ through a finite sequence of replacements by a collection of balls with respect to which $\Gamma$ has small modulus. This we call the ``weight reduction phase''. We construct a sequence of covers $\cB_k$, for $k\in \N$, as follows. Proceed inductively and apply Lemma \ref{lem:replaceannulus} for each $k\in \N$ with $\cB=\cB_k$, $\rho=\rho_k$ and $\cC=\cB_k$, and with $\rho_\bfB, \cB_\bfB$ and  $\epsilon=\eta$ satisfying \eqref{eq:smallanncond}, to obtain a collection $\cB_{k+1}=\overline{\cB}$ and function $\rho_{k+1}=\overline{\rho}$ with $\rho_{k+1}\overline{\wedge}_{\tau,\cB_{k+1}} \Gamma$ and
\[
\sum_{B\in \cB_{k+1}} \rho(B)^p \leq \epsilon^{k} M.
\]
We note that for each ball $B\in \cB_{\bfB}$ we have $\rad(B)\geq \delta_- \rad(\bfB)$, by the assumption of uniformly small moduli of annuli. Therefore, for all $k\in \N$ we get 
\begin{equation}\label{eq:stepradiusred1}
\inf\{\rad(B): B\in \cB_{k+1}\}\geq \delta_- \inf\{\rad(B): B\in \cB_k\}.
\end{equation}
By iteration of this inequality, we get
\begin{equation}\label{eq:stepradiusred}
\inf\{\rad(B): B\in \cB_{N}\}\geq \delta_-^N \inf\{\rad(B): B\in \cB_0\}.
\end{equation}
By the choice of $N$, we have $\epsilon^N M \leq \epsilon_0$. Thus, $\rho_{N}$ satisfies the desired mass bound:
\begin{equation}\label{eq:desiredmass}
\sum_{B\in \cB_{N}} \rho_N(B)^p \leq \epsilon_0.
\end{equation} 

The balls in $\cB_N$ have various different sizes. Next, we will embark on a ``size reduction phase''. Let $\overline{\cB}_0=\cB_N$, and $\overline{\rho}_0 = \rho_N$.
Let
$s:= \min\{\rad(B): B\in \cB_{N}\},$
and let $S_0=\rad(\overline{\cB}_0)$. From the assumption and \eqref{eq:stepradiusred}, we obtain
\begin{equation}\label{eq:radiuslowerbound}
s\geq \delta_-^{N} \inf\{\rad(B): B\in \cB_0\}\geq r.
\end{equation}
If $S_0 \leq \kappa r$, then we do not do anything and we let $L=0$. If on the other hand $S_0 > \kappa r$, we start running the following algorithm.

Set $k=0$. While $S_k>\kappa r$, let $\cC_k = \{B\in \overline{\cB}_k : \rad(B)>\kappa r\}$. Apply Lemma \ref{lem:replaceannulus} with $\cB=\overline{\cB}_k$, $\rho=\overline{\rho}_k$ and $\cC=\cC_k$ and with $\rho_\bfB, \cB_\bfB$ and with $\epsilon=\eta$ satisfying \eqref{eq:smallanncond}. This gives a collection $\overline{\cB}_{k+1}$ and strongly admissible function $\overline{\rho}_{k+1}$. Set $S_{k+1}=\rad(\overline{\cB}_{k+1})$, and increment $k$ by one. Once $S_{k}\leq \kappa r$, terminate the algorithm. We will soon show that the algorithm terminates in finite time. Let $L=k$ be the time it terminates.

We have, as part of Lemma \ref{lem:replaceannulus}, that $\overline{\rho}_k \overline{\wedge}_{\tau, \overline{\cB}_k} \Gamma$ for every $k\in [0,L]\cap \Z$. Further,  by noting that $\epsilon \in (0,1)$, we get
\[
\sum_{B\in \overline{\cB}_{k+1}} \overline{\rho}_{k+1}(B)^p \leq 
\sum_{B\in \overline{\cB}_{k+1}} \overline{\rho}_k(B)^p.
\]
By iterating this $k$ times, we get from \eqref{eq:desiredmass}  that
\begin{equation}\label{eq:desiredmass2}
\sum_{B\in \overline{\cB}_{k}} \overline{\rho}_k(B)^p \leq \sum_{B\in \overline{\cB}_{0}} \overline{\rho}_0(B)^p  \leq \ \epsilon_0.
\end{equation} 

Let us analyse the effect of the algorithm on the radii of the collections $\overline{\cB}_k$, and the termination of the algorithm. Assume that $k\geq 0$. At each step, a ball $B$ in  $\cB_{k+1}$ is either equal to a ball $\bfB\in\overline{\cB}_{k}$ with $\rad(\bfB)\leq \kappa r$, or $B\in \cB_{\bfB}$ for some $\bfB\in \overline{\cB}_k$ with $\kappa s < \rad(\bfB)\leq S_k$. By construction, in either case $\rad(B)\leq \max\{\delta_+ \rad(\bfB), s\kappa\}$. Thus, by taking a supremum over all balls $B\in \overline{\cB}_{k+1}$, we get that $S_{k+1}\leq \max\{\kappa r, \delta_+S_k\}$. In particular, while $S_k>\kappa r$, then the values $S_{k}$ form a geometrically decreasing sequence. This can only last for finitely many steps. Therefore, there must exist some $L\geq 0$, when the algorithm terminates with $S_L \leq \kappa r$.

We show now that by induction each ball $B\in \overline{\cB}_{k}$, for $k=0,\dots, L$ satisfies $\rad(B)\in [r,S_k]$. The upper bound is obvious, so we focus on the lower bound. The case of $k=0$ is also obvious. 
So, we focus on the induction step. During the algorithm, for $k=0,\dots, L-1$, each ball $B\in \overline{\cB}_{k+1}$ is either equal to a ball $\bfB\in \overline{\cB}_{k}$ or for some $\bfB\in \overline{\cB}_k$ we have $B\in\overline{\cB}_{\bfB}$ and $\rad(\bfB)>\kappa r$. 
In the first case, $\rad(B)\in [r,\kappa r]$. In the second case $\rad(B)\in [\delta_- \rad(\bfB),\delta_+\rad(\bfB)]$, and thus $\rad(B)\geq \delta_- \kappa r > r$ since $\delta_->\kappa^{-1}$ by choice of $\kappa$ at the beginning of the proof. In either case $r\leq \rad(B)\leq S_k$. 
Therefore, for all $B\in \overline{\cB}_{k}$, for $k=0,\dots, L$ we have $\rad(B)\in [r,S_k]$.

Now, for $k=L$, we have $\rad(B)\in [r,\kappa r]$, since $S_L\leq \kappa r$. Now set $\cV=\overline{\cB}_L$. We thus get the desired claim, since \eqref{eq:desiredmass2} gives the desired mass bound for $\overline{\Mod}_{p,\tau}(\Gamma,\cV)$, since $\overline{\rho}_L \overline{\wedge}_{\tau,\overline{\cB}_L} \Gamma$, and we have already observed that for all $V\in \cV$ we have $\kappa^{-1}r\leq r_V\leq r$.
\end{proof}

\subsection{Estimate for Bourdon-Kleiner modulus}

In this section, we use the algorithm in the previous subsection to give an explicit relationship between the Bourdon-Kleiner modulus from Subsection \ref{subsec:discretemoduli} and our new discrete modulus from Subsection \ref{subsec:newmod}. The basic idea is to use doubling to give an initial collection $\cV$, and then to use Lemma \ref{lem:equalizingalg} to push the collection down to roughly uniform size with small modulus. This push-down operation is quantitative. Once the collection consists of balls of roughly the same size, we can apply Proposition \ref{prop:discretemodulus} to compare the modulus to the Bourdon-Kleiner modulus of the same collection.


\begin{proposition} \label{prop:smallmodulus} Fix $\kappa\geq 1, p\in (1,\infty)$. For each $k\in \N$, let $\cU_{k}$ be a $\kappa$-approximations at scale $2^{-k}$ for a compact LLC space $X$.  If $X$ has uniformly small $p$-moduli of annuli, then for every $\epsilon>0$, there exists a $l\in \N$ for which for all $z\in X$ and all $k\geq 0$, we have
\[
\Mod_{p,\cU_{l+k}}(\Gamma_{B(z,2^{-k}),2})\leq \epsilon
\]
\end{proposition}
\begin{proof}
Fix $k\in \N$ and $\epsilon>0$. Let $\tau=4$ and let $l_0=\lceil\log_2(\tau)\rceil+4$. Let $X$ have uniformly small $p$-moduli of annuli with constant $\delta_-\in (0,\tau^{-1})$.  Choose $\kappa'\geq \kappa$ be the constant from Lemma \ref{lem:equalizingalg}, and let $C$ be the constant associated to $\kappa',\tau$ and the space $X$ which comes from Proposition \ref{prop:discretemodulus}. Set $\epsilon_0 = C^{-1}\epsilon$. 

By doubling, we have that there is a constant $D$ independent of $k$ so that there are at most $D$ many sets in $\cU_{k+l_0}$ which intersect $B(z,2^{1-k})$. Set
\[
\cB_0=\{B(x_U,r_U): U\in \cU_{k+l_0}, U\cap B(z,2^{1-k}) \neq \emptyset\}
\]
and set $\rho_0(B)=1$ for all $B\in \cB_0$. Then, by applying the definition, and since $\cB_0$ covers $B(z,2^{1-k})$, we see $\rho_0 \overline{\wedge}_{\tau,\cB_0} \Gamma_{B(z,2^{-k}),2}$. By the size bound for $\cB_0$, we get
\[
\sum_{B\in \cB_0} \rho_0(B)^p \leq D.
\]
By Lemma \ref{lem:equalizingalg}, there exists an integer $N\in \N$ (which depend on $\epsilon$, $D$ and the constants in the uniformly small moduli condition) with the following properties. For any $r>0$ with $\delta_-^{N} \inf\{\rad(B): B\in \cB_0\}\geq r$ there is a collection of balls $\cV$ with
\[
\overline{\Mod}_{p,\tau}(\Gamma, \cV) \leq \epsilon_0=C^{-1}\epsilon.
\]
and $r_V\in [\kappa'^{-1} r, r]$ for all $V\in \cV$.  

Now, if $l\geq l_0+N\lceil\log_2(\delta_-^{-1})\rceil+1$, then we can choose $r=2^{-k-l}$.
Then, by Proposition \ref{prop:discretemodulus}, we get for the $\kappa$-approximation $\cU_{l+k}$ at level $r$ that
\[
\Mod_{p,\cU_{l+k}}(\Gamma_{B(z,2^{-k}),2})\leq C\overline{\Mod}_{p,\tau}(\Gamma, \cV)\leq \epsilon.
\]
\end{proof}
\subsection{Proof of main theorem}\label{subsec:proofmain}

\begin{proof}[Proof of Theorem \ref{thm:mainthm}] Because the Ahlfors regular conformal dimension is always greater than the conformal Hausdorff dimension, we have $\dims_{CH}(X)\leq \dims_{CAR}(X)$. We are left to prove the converse inequality. Since $X$ is connected, compact, locally connected and quasiself-similar, by Lemma \ref{lem:quasiselfconn} $X$ is LLC.

Suppose that $p$ is arbitrary and $\dims_{CH}(X)<p$. Fix any sequence of $\kappa$-approximations $\{\cU_k\}_{k\in \N}$, where $\cU_k$ is at scale $2^{-k}$. By Lemma \ref{lem:Hausdorff}, we have that $X$ has uniformly small moduli of annuli. Then,  by Proposition \ref{prop:smallmodulus}, we have that 
\[
\liminf_{m\to\infty} \sup_{x\in X, k\in \N}\{ \Mod_p(\Gamma_{B(x,2^{-k}),2},\cU):  \cU \text{ is a } \kappa-\text{ approximation at level } 2^{-k-m}\} = 0.
\]
Then, Theorem \ref{thm:confhausdim} implies that $\dims_{CAR}(X)\leq p$. Since $p>\dims_{CH}(X)$ is arbitrary, this completes the proof. 

\end{proof}

\bibliographystyle{acm}
\bibliography{pmodulus}

\end{document}